\documentclass[reqno]{amsart}
\usepackage{graphicx}
\usepackage{pb-diagram}
\usepackage{pstricks}
\usepackage[all]{xy}
\theoremstyle{theorem}
\newtheorem{theorem}{Theorem}[section]
\newtheorem{corollary}[theorem]{Corollary}
\newtheorem{lemma}[theorem]{Lemma}
\newtheorem{prop}[theorem]{Proposition}

\theoremstyle{definition}
\newtheorem{definition}{Definition}[section]
\newtheorem{remark}[definition]{Remark}

\newtheorem{ques}[definition]{Question}
\numberwithin{equation}{section}

\newcommand{\CC}{\mathbb{C}}
\newcommand{\RR}{\mathbb{R}}
\newcommand{\QQ}{\mathbb{Q}}
\newcommand{\ZZ}{\mathbb{Z}}
\newcommand{\NN}{\mathbb{N}}

\newcommand{\OM}{\overline{m}}
\newcommand{\FC}{\mathfrak{C}}
\newcommand{\CF}{\mathcal{F}}

\newcommand{\FA}{\mathfrak{A}}
\newcommand{\FF}{\mathfrak{F}}
\newcommand{\FN}{\mathfrak{N}}
\newcommand{\FM}{\mathfrak{M}}

\newcommand{\CM}{\mathcal{M}}

\newcommand{\CE}{\mathcal{E}}

\newcommand{\AI}{A_\infty}
\newcommand{\LI}{L_\infty}

\newcommand{\WH}[1]{\widehat{\overline{#1}}}
\newcommand{\HH}[1]{\widehat{#1}}

\newcommand{\WT}[1]{\widetilde{#1}}
\newcommand{\OL}[1]{\overline{#1}}
\newcommand{\UL}[1]{\underline{#1}}
\newcommand{\E}{\epsilon}
\newcommand{\kk}{\boldsymbol{k}}
\newcommand{\NOV}{\Lambda_{nov}}
\newcommand{\NOVO}{\Lambda_{0,nov}}
\newcommand{\NOVE}{\Lambda_{nov}^{(e)}}
\begin{document}
\title{On the obstructed Lagrangian Floer theory}
\author{Cheol-Hyun Cho}
\address{Department of Mathematics, Seoul National University,
Kwanakgu Shinrim, San56-1 Seoul, South Korea, Email:
chocheol@snu.ac.kr}

\begin{abstract}
Lagrangian Floer homology in a general case has been constructed by Fukaya, Oh, Ohta and Ono, where they construct an $\AI$-algebra or an $\AI$-bimodule from Lagrangian submanifolds, and studied their obstructions and deformation
theories.  But for obstructed Lagrangian submanifolds, the standard Lagrangian Floer homology can not be defined. 

We explore several well-known cohomology theories on these $\AI$-objects and explore their properties, which are well-defined and invariant even in the obstructed cases.  These are Hochschild and cyclic homology of an $\AI$-objects and Chevalley-Eilenberg or cyclic Chevalley-Eilenberg homology of
their underlying $\LI$ objects. We explain how the existence of $m_0$ effects the usual homological algebra of these homology theories.
We also provide some computations. We show that for an obstructed $\AI$-algebra with a non-trivial primary obstruction, Chevalley-Eilenberg Floer homology vanishes, whose proof is inspired by the comparison with cluster homology theory of Lagrangian submanifolds by Cornea and Lalonde.

In contrast, we also provide an example of an obstructed case whose cyclic Floer homology is non-vanishing.
\end{abstract}
\thanks{This work was supported by the Korea Research Foundation Grant funded by the Korean Government (MOEHRD, Basic Research Promotion Fund) (KRF-2008-C00031)}
\maketitle

\section{Introduction}
Floer homology invented by Floer \cite{Fl}, has proven to be a very
powerful tool in the symplectic geometry and related areas. In the case of Floer cohomology of Lagrangian submanifolds, Fukaya, Oh,
Ohta and Ono \cite{FOOO} has defined the $\AI$-algebra of a Lagrangian
submanifold or the $\AI$-bimodule of a pair of Lagrangian
submanifolds in the full generality. Unlike the case of Hamiltonian Floer cohomology, these $\AI$-algebra or $\AI$-bimodules do not always define
Floer homologies as they are obstructed in general. In \cite{FOOO}, they have studied the obstruction theory and showed that if the obstructions vanish, these $\AI$-structures give rise to the Floer cohomology theories, and can be applied to the study of 
symplectic topology or homological mirror symmetry (see \cite{K2}, \cite{CO}, \cite{FOOO}, \cite{FOOO2} for example).

But if a Lagrangian submanifold is obstructed (see section 2 and 7 for its definition), its Floer cohomology is not defined. In this paper, we explore
alternative ways of defining homology, by considering Hochschild and cyclic homology of an $\AI$-objects and Chevalley-Eilenberg or cyclic Chevalley-Eilenberg homology of their underlying $\LI$ objects.

Such homology theories are well-known for associative (or Lie) algebras and also for $\AI$, or $\LI$-algebras without $m_0$.
(We refer readers to \cite{HL} for the definitions using non-commutative geometry). The definition easily extends to the
case with with $m_0$, but it turns out that the usual homological algebra properties of these homology theories do not
immediately extend as the usual contracting homotopy of the bar complex does not work with $m_0 \neq 0$. But by working with Novikov fields, we
study their homological algebras and show that we still have the reduced Hochschild homology, and $(b,B)$-cyclic complex where Conne-Tsygan $B$-operator actually has an additional term compared to the standard case.

The main motivation to study these homology theories is to have a well-defined Floer homology theory even in the obstructed cases.
We show that even in the obstructed cases, these homology theories are well-defined and invariant under various choices involved  and 
define invariants of a homotopy class of $\AI$-objects. 

We remark that there has been different approaches to consider obstructed cases, by Cornea and Lalonde \cite{CL} using Morse functions and by Fukaya \cite{Fu} using the relationship with loop space homology and Floer homology.

We observe that the topological dual theory of the Chevalley-Eilenberg homology theory of $\LI$-objects is related to 
the cluster homology theory announced by Cornea and Lalonde in \cite{CL}. Their cluster homology corresponds to the extended
cyclic Chevalley-Eilenberg homology of $\LI$-algebras and their symmetric fine Floer homology corresponds to Chevalley-Eilenberg homology
of $\LI$-modules. Unfortunately, analytic details of the construction of cluster homology theory in \cite{CL} has not been rigorously established yet,
but the homology theories in this paper may provide an alternative way to consider such theories in the obstructed cases.

We remark that the study of the dual geometry of an $\AI$ or $\LI$-algebras which was initiated by Kontsevich \cite{K1}.
The $\AI$-algebra, which is a coalgebra has a dual which can be regarded as a differential graded algebra (DGA). 
Actually, in contact geometry, the dual language has been mostly used (see Chekanov \cite{Che}, Eliashberg-Givental-Hofer \cite{EGH} for example)
and the analogue of unobstructed condition is the notion of augmentation. In both cases, if it is unobstructed or has an augmentation, the deformed
homology theory provides much more refined informations.

Due to the use of Novikov coefficients, to take the appropriate dual of the filtered $\AI$ or $\LI$ objects, we consider topological duals induced by energy filtrations.
In fact, the completion used in this paper is somewhat different from that of Cornea and Lalonde, resulting different behavior of the 
homology theories in obstructed cases. But this does not bring a major difference as Cornea has informed me that the filtration used here also works in the cluster homology setting.

There is an easy way to obtain cyclic homology complex of an $\AI$-algebra or a cyclic Chevalley-Eilenberg homology complex from the bar complex.
Recall that $\AI$-algebra $C$, which is given by countably many
operations $\{m_k\}$, is algebraically a tensor coalgebra $T(C[1])$
with a codifferential $\HH{d} = \sum_k \HH{m}_k$. The complex $(T(C[1]),\HH{d})$ is called a bar complex, whose
homology is trivial(see Lemma \ref{bartrivial}). One can consider cyclic or symmetric bar
complex, which is a subcomplex of the bar complex by considering the
fixed elements of the natural cyclic or symmetric group action. The homology of these
subcomplexes are in fact the cyclic homology of $\AI$-algebra or cyclic Chevalley-Eilenberg homology of the induced $\LI$-algebra.
Here, as any associative algebra can be regarded as a Lie algebra whose
bracket is given by the commutator, an $\AI$-algebra ($\AI$-module) gives rise to an underlying $\LI$-algebra ($\LI$-module) by symmetrizing all $\AI$-operations and Chevalley-Eilenberg homology is their Lie algebra homology. 
We remark that the existence and invariance of homology of these subcomplexes has been known to authors of \cite{FOOO}.
(Note that what we call cyclic bar complex is different from the cyclic bar complex of Getzler-Jones \cite{GJ})

We will also consider Hochschild homology of $\AI$-bimodule of a pair of Lagrangian submanifolds when one Lagrangian submanifold is obtained as a Hamiltonian isotopy of the other and this homology contains information about 
their intersections. Similarly, one can consider the induced $\LI$-module of such pair over the $\LI$-algebra of such a Lagrangian submanifold and
consider its Chevalley-Eilenberg homology.

\begin{theorem}
If $A$ is an obstructed $\AI$-algebra of Lagrangian submanifold with non-trivial primary obstruction, then the Chevalley-Eilenberg Floer (co)homology,
and the extended cyclic Chevalley-Eilenberg Floer cohomology vanish.
\end{theorem}
Recall that in \cite{CL}, cluster complex with free terms has vanishing cluster homology. This may be interpreted as the vanishing of Chevalley-Eilenberg Floer homology when $m_0 \neq 0$ if we use the filtration of \cite{CL}. Hence one can notice the subtlety in choosing filtrations.

In contrast, we find a very different phenomenon of cyclic homology of $\AI$-algebra in some cases.
\begin{theorem}\label{thm:2}
Let $L$ be a relatively spin compact Lagrangian submanifold in a closed (or with convex boundary) symplectic manifold, which admits only non-positive Maslov index pseudo-holomorphic discs with boundary on $L$. 
Then, its cyclic homology of the $\AI$-algebra of $L$ is non-trivial even when $L$ is obstructed.
\end{theorem}
The proof of the above theorem relies on the construction of explicit non-vanishing element of the
cyclic homology in such a case.

\begin{theorem}
Let $L$ be a relatively spin compact Lagrangian submanifold in a closed (or with convex boundary) symplectic manifold, which is
displaceable by a Hamiltonian isotopy. Here $\AI$-algebra of $L$ may be obstructed.
Then, its Hochschild homology of the $\AI$-algebra of $L$ vanishes and
also the Chevalley-Eilenberg homology vanishes
\end{theorem}

Now, the following question is related to the Maslov class conjecture in the obstructed case, but
we do not know whether it is true or not.
\begin{ques}
Let $L$ be a relatively spin compact Lagrangian submanifold in a closed (or with convex boundary) symplectic manifold, which is
displaceable by a Hamiltonian isotopy. Does the cyclic homology of the $\AI$-algebra of $L$ vanish ?
Also does the cyclic Chevalley-Eilenberg homology vanish?
\end{ques}
If the answer for the first question is yes, then together with the theorem \ref{thm:2} it should prove the Maslov class conjecture even for
the obstructed relatively spin compact Lagrangian submanifolds in $\CC^n$ that its Maslov class vanishes
(see \cite{O},\cite{P},\cite{V},\cite{Fu2}) and we leave this for future research.

This paper is organized in the following way. In section 2, we recall basic
notions of $\AI$-algebra, $\AI$-bimodule, and their cyclic and
symmetric versions. We also recall various notions from \cite{FOOO} which is
needed to prove the results. In section 3, we show the isomorphism property of  related homology theories
for homotopy equivalent objects.
In section 4, we explain 
the Hochschild and Chevalley-Eilenberg homology and isomorphisms under weakly filtered homotopy equivalences.
In section 5, we consider cyclic version of the theories in section 4, and show the modification of homological algebra
with the presence of $m_0$.
In section 6, we recall results of \cite{FOOO}, and apply the discussed homology theories.
In section 7, we explain the relation between Maurer-Cartan element and Hochschild homology and augmentation.
In section 8, we find a non-trivial element in the cyclic Floer homology.
In section 9, we consider topological dual theories of the above and in section 10, we consider dualizations of Chevalley-Eilenberg homology and show its comparison to cluster homology theory of \cite{CL}.

{\bf Acknowlegements} We thank Kenji Fukaya, Octav Cornea, Yong-Geun Oh and Kaoru Ono for helpful communications.

\section{Algebraic setup}
We briefly recall the notions about $\AI$-algebras and $\AI$-bimodules, and their cyclic and symmetric versions.
We also recall gapped condition and $\AI$-homotopy from \cite{FOOO} to which we refer readers for details.

\subsection{$\AI$ and $\LI$-algebras}
Let $R$ be the field $\RR$. We can also consider instead $\CC$ or $\QQ$. But $\QQ$ can be used except in the last section
due to the Theorem \ref{thm:derham}.

Let $\OL{C}=\bigoplus_{j\in \mathbb{Z}} \OL{C}^{j}$ be a graded
vector space over $R$.
We denote the parity change(or suspension) as $( \OL{C}[1])^m =  \OL{C}^{m+1}$, and 
$|x_i|$ is the degree of the element $x_i$, and $|x_i|'$ is the shifted degree.
Hence $|x_i|=|x_i|'+1$.
We define
\begin{equation}\label{def:tk}
T_k(\OL{C}[1]) = \underbrace{\OL{C}[1] \otimes \cdots \otimes \OL{C}[1]}_{k},\;
 T_{1,\cdots,k}(\OL{C}[1]) = \oplus_{j=1}^k T_j(\OL{C}[1]).
\end{equation}
To simplify the notation, we set
$$B_k(C):=T_k(\OL{C}[1]), B_{1,\cdots,k}(C) =  T_{1,\cdots,k}(\OL{C}[1]).$$
\begin{definition}  The \textbf{tensor-coalgebra}
of $\OL{C}[1]$ over $R$ is given by
$ B\OL{C}:=\bigoplus_{k\geq 1} T_k(\OL{C}[1]),$ with the comultiplication defined by
$$ \Delta:B\OL{C}\longrightarrow B\OL{C}\otimes B\OL{C}, \,\,\,\,\,\,\,\,\,
   \Delta(v_{1}\otimes \cdots \otimes v_{n}):=\sum_{i=1}^{n} (v_{1}\otimes \cdots \otimes v_{i})
    \otimes(v_{i+1}\otimes \cdots \otimes v_{n}).$$
\end{definition}
Now, consider a family of maps
$$\OL{m}_k : T_k(\OL{C}[1]) \to \OL{C}[1], \ \ \textrm{for} \ k =1,2,\cdots.$$
We can extend $\OL{m}_k$ uniquely to a coderivation
\begin{equation}\label{eq:hatd}
\WH{m}_k(x_1 \otimes \cdots \otimes x_n) = \sum_{i=1}^{n-k+1}
(-1)^{|x_1|' + \cdots + |x_{i-1}|'} x_1 \otimes \cdots \otimes m_k(x_i,
\cdots, x_{i+k-1}) \otimes \cdots \otimes x_n
\end{equation}
for $k \leq n$ and $\WH{m}_k(x_1 \otimes \cdots \otimes x_n) =0$ for $k >n$.

The coderivation $\WH{d} = \sum_{k=1}^\infty \WH{m}_k$ is well-defined as a map
from $B\OL{C}$ to $B\OL{C}$.
The $\AI$-equations are equivalent to the equality
$\WH{d} \circ \WH{d} =0$, or equivalently,

\begin{definition}\label{def:ai}
 An $\AI$-algebra $(\OL{C},\{m_*\})$ consists of a $\ZZ$-graded vector space $\OL{C}$ over $R$ with
a collection of multi linear maps  $m:= \{ m_n:\OL{C}[1]^{\otimes n} \to \OL{C}[1]\}_{n\geq 1}$ of
degree one
 satisfying the following equation for each $k=1,2,\cdots.$
\begin{equation}\label{aiformula}
0=\sum_{k_1+k_2=k+1} \sum_{i=1}^{k_1-1} (-1)^{\epsilon_1} m_{k_1}(x_1,\cdots,x_{i-1},m_{k_2}(x_i,\cdots,x_{i+k_2-1}),\cdots,x_k)
\end{equation}
where $\epsilon_1= |x_1|'+ \cdots + |x_{i-1}|'$.
\end{definition}

Now, we explain the definition of an $\LI$-algebra.
First, consider an element $\sigma$ of the group $S_k$ of all permutations of the set $\{1,2,\cdots,k\}$.
The group $S_k$ act on $T_k(\OL{C}[1])$ by
\begin{equation}\label{gaction}
\sigma \cdot (x_1 \otimes \cdots \otimes x_k) = (-1)^{\E(\sigma,\vec{x})} x_{\sigma(1)} \otimes \cdots
\otimes x_{\sigma(k)},
\end{equation}
where
\begin{equation}\label{signperm}
\E(\sigma,\vec{x}) = \sum_{i,j \ \textrm{with} \ i<j,\ \sigma(i)>\sigma(j)}
(|x_i|' \cdot |x_j|').
\end{equation}
For example, if we denote the cyclic element $\sigma_0 =(1,2,\cdots,k) \in S_k$, note that
$$\sigma_0 \cdot (x_1\otimes \cdots \otimes x_k)= (-1)^{(|x_1|'(\sum_{i=2}^{k}|x_i|'))} x_2 \otimes \cdots \otimes x_{k} \otimes x_1.$$
\begin{definition}
Let $B^{cyc}_k\OL{C}$ be the set of fixed elements of the above $\sigma_0$ action on $B_k\OL{C}$, and denote 
$$ B^{cyc}\OL{C} = \oplus_{k=1}^\infty B^{cyc}_k\OL{C}.$$
\end{definition}
\begin{definition}
We define $E_k(\OL{C})$ to be the submodule of $B_k(\OL{C})$ consisting of
{\it fixed} elements of $S_k$-action on $B_k(\OL{C})$ and
let $$E\OL{C} = \oplus_{k=1}^\infty E_k(\OL{C}).$$
\end{definition}
\begin{remark}
Equivalently, one can instead use the quotient complex by defining the equivalence relation by the above cyclic or symmetric group action.
\end{remark}
For convenience, we use the following notation for the generators of $E\OL{C}$:
$$[x_1,\cdots,x_k]= \sum_{\tau \in S_k} (-1)^{\E(\tau,\vec{x})} x_{\tau(1)}\otimes \cdots \otimes x_{\tau(k)}.$$
One can easily check that
\begin{lemma}\label{lem:coalge}
There is  a coalgebra structure on $E\OL{C}$
$$\Delta: E\OL{C} \to E\OL{C} \otimes E\OL{C},$$ induced from $(B\OL{C},\Delta)$. This is graded commutative
and coassociative.
\end{lemma}
\begin{remark}
The above coalgebra
structure will be used to define commutative algebra structure to the dual space of $E\OL{C}$.
Note that $B^{cyc}\OL{C}$ does {\em not} have an induced coalgebra structure.
\end{remark}
\begin{proof}
One can see that 
$$\Delta \big( \sum_{\sigma \in S_k} (-1)^{\E(\sigma,\vec{x})} x_{\sigma(1)} \otimes \cdots
\otimes x_{\sigma(k)} \big)$$
$$=\sum_{\sigma \in S_k, i} (-1)^{\E(\sigma,\vec{x})} \big( x_{\sigma(1)} \otimes \cdots \otimes x_{\sigma(i)}
\big) \otimes \big( x_{\sigma(i+1)} \otimes \cdots \otimes x_{\sigma(k)} \big)$$
$$=\sum_i \sum_{\sigma \in (i,k-i) shuffle}
 (-1)^{\E(\sigma,\vec{x})} \big([ x_{\sigma(1)} \otimes \cdots \otimes x_{\sigma(i-1)}]
\big) \otimes \big( [x_{\sigma(i+1)} \otimes \cdots \otimes x_{\sigma(k)}] \big)$$
Define $T:E\OL{C} \otimes E\OL{C} \to E\OL{C} \otimes E\OL{C}$
by $$T(\alpha \otimes \beta) = (-1)^{|\alpha|'|\beta'|} \beta \otimes \alpha,$$
where $\alpha,\beta$ are homogeneous elements of degree $|\alpha|',|\beta'|$ respectively.
Then, it is not hard to see that $T \circ \Delta = \Delta$ which proves the
cocommutativity.
\end{proof}
\begin{lemma}[\cite{FOOO},\cite{Fu}]
The codifferential $\WH{d}$  descends to the map $B^{cyc}\OL{C} \to B^{cyc}C$.
Also $\WH{d}$ induces a codifferential $\WH{d}: E\OL{C} \to E\OL{C}$.
\end{lemma}
\begin{proof}
It is easy check the claim for the generators $\sum_{\sigma} \sigma(x_1 \otimes \cdots \otimes x_k)$
where the summand is over $\sigma \in \ZZ/k\ZZ$ for the cyclic case or $\sigma \in S_k$ for the symmetric case.
\end{proof}

$\LI$-algebra structure on $\OL{C}$ is the codifferential $\WH{d}$ on $E\OL{C}$ or equivalently,
\begin{definition}\label{def:linf}
 An $\LI$-algebra $(\OL{C},\{l_*\})$ consists of a $\ZZ$-graded vector space $\OL{C}$ over $R$ with
a collection of multi linear maps  $l:= \{ l_n:E_n\OL{C} \to \OL{C}[1]\}_{n\geq 1}$ of
degree one
 satisfying the following equation for each $k=1,2,\cdots.$
\begin{equation}
0=\sum_{k_1+k_2=k+1} \sum_{\sigma \in (k_1,k_2)\textrm{shuffle}} (-1)^{\E(\sigma,\vec{x})} 
l_{k_1}\big( [l_{k_2}([x_{\sigma(1)},\cdots,x_{\sigma(k_2)}]),x_{\sigma(k_2+1)},\cdots,x_{\sigma(k)}] \big).
\end{equation}
\end{definition}
Note that for any $\AI$-algebra, there exists the underlying $\LI$-algebra obtained by the restriction to
fixed elements of symmetric group action.

\subsection{$\AI$ and $\LI$-modules}
\begin{definition} For graded vector spaces $\OL{C}_1,\OL{C}_0$ and $\OL{M}$ over $R$, one writes
\begin{equation}\label{def:modbar}
T^{\OL{M}}(\OL{C}_1,\OL{C}_0):=\bigoplus_{k\geq 0, l\geq 0} \OL{C}_1^{\otimes k} \otimes \OL{M}
\otimes \OL{C}_0^{\otimes l}.
\end{equation}
Furthermore, let
$$ \Delta^{\OL{M}}:T^{\OL{M}}(\OL{C}_1,\OL{C}_0)\longrightarrow (T\OL{C}_1\otimes  T^{\OL{M}}(\OL{C}_1,\OL{C}_0))\oplus
(T^{\OL{M}}(\OL{C}_1,\OL{C}_0)\otimes T\OL{C}_0), $$ be given by
$$\Delta^{\OL{M}}(v_{1}\otimes \cdots \otimes v_{k}  \otimes  \UL{w}  \otimes v_{k+1}\otimes \cdots \otimes v_{k+l}):=$$
$$ \sum_{i=1}^{k}(v_{1}\otimes \cdots \otimes v_{i}) \otimes(v_{i+1}\otimes \cdots \otimes \UL{w} \otimes \cdots \otimes v_{n})+
  \sum_{i=k}^{k+l-1} (v_{1}\otimes \cdots \otimes \UL{w} \otimes \cdots \otimes v_{i})
 \otimes(v_{i+1}\otimes \cdots \otimes v_{k+l}).$$
Here we underlined the element of the module $\OL{M}$ for convenience.
\end{definition}

Let $(\OL{C}_1,\WH{d})$ and $(\OL{C}_0,\WH{d} ')$ be $\AI$-algebras. An $\AI$-bimodule $\OL{M}$ over
$(\OL{C}_1,\OL{C}_0)$ is defined as a map $D^M$ defined as follows:
for simplicity, we first denote  
$$\OL{B}^M(C_1,C_0) = T^{\OL{M}[1]}(\OL{C}_1[1],\OL{C}_0[1]).$$
We consider a map $D^{M}:\OL{B}^M(C_1,C_0) \longrightarrow \OL{B}^M(C_1,C_0)$ be a map of degree one
satisfying the following commutative diagram:
\begin{equation}
\begin{diagram}
\node{\OL{B}^M(C_1,C_0)}\arrow{s,l}{D^{M}}\arrow{e,t}{\Delta^{M}}
  \node{(B\OL{C}_1\otimes \OL{B}^M(C_1,C_2))\oplus (\OL{B}^M(C_1,C_0)\otimes B\OL{C}_0)}\arrow{s,r}
  {(id\otimes D^{M}+\WH{d}\otimes id) \oplus (D^{M}\otimes id+id\otimes \WH{d}')} \\
\node{\OL{B}^M(C_1,C_0)}\arrow{e,b}{\Delta^{M}} \node{
(B\OL{C}_1\otimes \OL{B}^M(C_1,C_0))\oplus (\OL{B}^M(C_1,C_0)\otimes B\OL{C}_0) }
\end{diagram}
\end{equation}
One can show (see \cite{T}) that such $D^M$ is determined by the family of maps
$$\eta_{k_1,k_0}:T_{k_1}\OL{C}_1[1] \otimes \OL{M}[1] \otimes T_{k_0}\OL{C}_0[1] \to \OL{M}[1].$$ 
We say $\OL{M}$ has the structure of an $\AI$-bimodule over $(\OL{C}_1,\OL{C}_0)$ if 
$$D^{M} \circ D^{M} =0.$$

Now, the definitions of $\LI$-bimodule can be obtained by symmetrizing the above construction.
Namely, let $(\OL{C},\WH{d})$ be an $\AI$-algebra and 
consider an $\AI$-bimodule $\OL{M}$ over $(\OL{C},\OL{C})$.
We can consider a symmetric group action on $\OL{B}^M(C,C)$ defined by
$$\sigma \cdot (x_{1}\otimes \cdots \otimes x_{k}\otimes \underline{x_{k+1}}\otimes x_{k+2}\otimes \cdots \otimes x_{k+l+1})
= $$
$$ (-1)^{\E(\sigma,\vec{x})} 
(x_{\sigma(1)}\otimes \cdots \otimes  \underline{x_{\sigma(j)}}\otimes  \cdots \otimes v_{\sigma(k+l+1)}),
$$
where $\sigma(j)=k+1$. 
For example,  $\sigma_0 = (1,2,3) \in S_3$,
$$\sigma_0 \cdot ( x_1 \otimes \UL{x_2} \otimes x_3) = \pm \UL{x_2} \otimes x_3 \otimes x_1.$$

Let $\OL{E}^M(C)$ be the fixed elements of the symmetric group action on $\OL{B}^M(C,C)$.
We denote 
$$[\UL{x}_0,x_1,\cdots,x_k]:= \sum_{\tau \in S_{k+1}, j=\sigma^{-1}(0)} (-1)^{\E(\tau,\vec{x})} x_{\tau(1)}\otimes \cdots 
\otimes \UL{x}_{\sigma(j)} \otimes \cdots \otimes x_{\tau(k)}.$$
Note that there exist obvious one to one correspondence between $\OL{E}^M(C)$ and $\OL{M}[1] \otimes E_k\OL{C}$, and we will identify them for
convenience.

Let $\WT{\OL{C}}$ be an induced $\LI$-algebra from the $\AI$-algebra $\OL{C}$. 
We define $\LI$-bimodule $\OL{M}$ over an $\LI$-algebra $\WT{\OL{C}}$ as a map 
$$D^{M}:\OL{E}^M(C) \longrightarrow \OL{E}^M(C)$$ obtained by symmetrizing the above construction such that $D^{M} \circ D^{M} =0$. Or equivalently, 

\begin{definition}\label{def:lmod}
Let  $(\OL{C},\{l_*\})$  be an $\LI$-algebra. Then, the $\LI$-bimodule structure on a graded vector space$\OL{M}$
is given by a collection of maps  $\eta:= \{ \eta_k: \OL{M}[1]\otimes E_k\OL{C} \to \OL{C}[1] \}_{k\geq 0}$ of
degree one  satisfying the following equation for each $k=0,1,2,\cdots$
\begin{equation}
\sum_{k_1+k_2=k} \sum_{\sigma \in (k_1,k_2)\textrm{shuffle}}(-1)^{\E(\sigma,\vec{x}_{\geq 1})}  \big( 
\eta_{k_1} ([ \eta_{k_2}[ \UL{x}_0, x_{\sigma(1)},\cdots,x_{\sigma(k_2)}] ,x_{\sigma(k_2+1)},\cdots,x_{\sigma(k)}]) 
\end{equation}
$$+ \eta_{k_1} ([ \UL{x}_0, l_{k_2}[x_{\sigma(1)},\cdots,x_{\sigma(k_2)}] ,x_{\sigma(k_2+1)},\cdots,x_{\sigma(k)}])
\big) =0.$$
\end{definition}

In fact, the notions such as $\AI$-homomorphism, $\AI$-bimodule map, $\AI$-homotopy induces those
of $\LI$-homomorphism, $\LI$-bimodule map, $\LI$-homotopy without much difficulty via the process of
symmetrization. These notions will be explained more in the setting of filtered $\AI$-algebras now.

\subsection{Filtered $\AI$-algebras}
We fix the ring $R=\RR$, and we use the following Novikov rings as  coefficients. 
($T$ and $e$ are formal parameters)
$$\NOV = \{ \sum_{i=0}^\infty a_i T^{\lambda_i}e^{q_i} | \; a_i \in R,\;\lambda_i \in \RR,\; q_i \in \ZZ, \; \lim_{i \to \infty} \lambda_i = \infty \}$$
$$\NOVO = \{ \sum_i a_i T^{\lambda_i}e^{q_i} \in \NOV | \lambda_i \geq 0 \},\;\;
\NOVO^+ = \{ \sum_i a_i T^{\lambda_i}e^{q_i} \in \NOVO | \lambda_i > 0 \},$$
We define a valuation $\tau:\NOV \to \RR$ which is the minimum energy of the expression:
\begin{equation}\label{def:tau}
\tau(\sum_i a_i T^{\lambda_i}e^{q_i}) = Min ( \{ q_i|\forall i\; \}).
\end{equation}
We also define the energy filtration as
$$ F^{\lambda} \NOVO =  \{ x \in \NOVO | \tau(x) \geq \lambda \}.$$
We remark that in \cite{FOOO}, they work with $\NOVO$ coefficient to define these $\AI$-objects and as $\NOV$ is flat over $\NOVO$, it does not cause any trouble. But one should be careful since the Floer cohomology of a pair of Lagrangian submanifolds over $\NOV$ is  invariant under
the Hamiltonian isotopy, but not over $\NOVO$. This is because the related maps are
weakly filtered $\AI$-bimodule maps which will be explained later in this section.
We remark that unfiltered $\AI$-algebras as in the previous subsections are denoted as $\OL{C}$ and their $\AI$-maps as $\OL{m}_i$'s.

We would like to work on Novikov field coefficients to the following cases: first, to find the reduced Hochschild homology. Second, when we try to construct the Conne-Tsygan's $B$ operator  and finally
when we  make dualizations to obtain $DGA$'s. But note that $\NOV$ is not a field due to the formal parameter $e$.
For example, $(1+e)$ is not an invertible element. To overcome this, we may take one of the following two approaches.
First, one define universal Novikov ring without $e$:
\begin{equation}\label{def:lambda}
\Lambda =\{ \sum_{i=0}^\infty a_i T^{\lambda_i} | \; a_i \in R,\;\lambda_i \in \RR,\; \; \lim_{i \to \infty} \lambda_i = \infty \}
\end{equation}
In this case, $\Lambda$ is a field, but as one lose the track of the index, one should work with $\ZZ/2$-graded complexes instead of $\ZZ$ graded ones except for the cases of vanishing Maslov index. Here there exists at least $\ZZ/2$-grading as the Maslov index of a holomorphic disc with Lagrangian boundary conditions is always even for an orientable Lagrangian submanifold.
(We learned this approach from Fukaya and this will be used in their upcoming work on toric manifolds.)

On the other hand, to keep $e$ alive, we can also take the following approach. We consider a field of rational functions $R(e)$ of the variable $e$, and
consider the tensor product
\begin{equation}\label{def:nove}
\NOV^{(e)} = \NOV \otimes_{R[e,e^{-1}]} R(e).
\end{equation}
Then, we obtain a field $\NOV^{(e)}$ and as it is obtained via tensoring the field $R(e)$, it does not affect the homology theories very much.
We remark that in most of the construction of \cite{FOOO}, they work with $\NOVO$ and only when one needs to work with $\NOV$, they take tensor product  $\otimes \NOV$ to work with
$\NOV$ coefficients. We will take a similar approach when using the field  $\NOV^{(e)}$.

Let $C$ be a free graded $\NOVO$-module.
We define similarly $T_k(C[1]), B_k(C)$ for $k\geq 1$ as in (\ref{def:tk}), and
set $T_0(C[1])= B_0(C)= \NOVO$. 
A filtered $\AI$-algebra structure on $C$ is defined as in the Definition \ref{def:ai}
by a family of maps
$$m_k : B_k(C) \to C[1], \ \ \textrm{for} \ k =0,1,\cdots,$$
satisfying the $\AI$-equations (\ref{aiformula}).

The module $C$ also has a filtration from the filtration of $\NOVO$.
The filtration on $B_k(C)$ is defined as
$$F^\lambda B_k(C) = \cup_{\lambda_1+ \cdots +\lambda_k \geq \lambda}
\big( F^{\lambda_1} C \otimes \cdots \otimes F^{\lambda_k} C \big).$$
Define$$BC = \oplus_{k=0}^\infty B_k(C),$$ and $\HH{B}C$ be its completion
with respect to energy filtration.

Then, one can define $\HH{m}_k$ as in (\ref{eq:hatd}), and note that when $k=0$, $\HH{m}_0$ is defined as
$$\HH{m}_0(x_1 \otimes \cdots \otimes x_n) = \sum_{i=1}^{n-k+1}
(-1)^{|x_1|'+\cdots+|x_{i-1}|' } x_1 \otimes \cdots \otimes
x_{i-1} \otimes m_0(1) \otimes \cdots \otimes x_n.$$
Then by setting $\HH{d} = \sum_{k=0}^\infty \HH{m}_k$,
the $\AI$-equations are equivalent to the equality
$\HH{d} \circ \HH{d} =0$.
The complex $(\HH{B}C,\HH{d})$ is called the {\em bar complex} of an $\AI$-algebra $A$.

The above equality give rise to countably many relations among $\{ m_k \}$ where
the first two are given as
\begin{equation}\label{dd}
\begin{cases}
 m_1(m_0(1)) =0.\\
m_2(m_0(1),x) + (-1)^{deg \ x +1}m_2(x,m_0(1)) + m_1(m_1(x)) =0.
\end{cases}
\end{equation}
If $m_0 =0$, we have $m_1^2 =0$, hence it defines the homology of the $\AI$-algebra.
In general $m_0$ does not vanish, hence $m_1$ is not necessarily a differential.
The obstruction and deformation theory when $m_0 \neq 0$ was developed in \cite{FOOO},
and in an unobstructed case, one can define (deformed) Floer cohomology. See section 7 for
more discussion on unobstructedness.

An element $I \in C^0 = C^{-1}[1]$ is called a unit if
\begin{equation}\label{unit}
\begin{cases}
 m_{k+1}(x_1,\cdots,I,\cdots,x_k) = 0 \;\; \textrm{for} \; k\geq 2 \;\;\textrm{or}\;\; k =0 \;\; \\
m_2(I,x) = (-1)^{deg \, x} m_2(x,I) = x.
\end{cases}
\end{equation}

For filtered $\AI$-algebras, we assume the maps $\{m_k \}$ satisfy
\begin{equation}\label{menergy}
\begin{cases}
m_k \big( F^{\lambda_1}C^{m_1} \oplus \cdots \oplus F^{\lambda_k}C^{m_k}\big) \subset F^{\lambda_1 + \cdots + \lambda_k}C^{m_1 + \cdots + m_k -k +2} \\
m_0(1) \in F^{\lambda'}C[1] \,\;\; \textrm{for some}\; \lambda' >0.
\end{cases}
\end{equation}
We remark that $$\NOVO / \NOVO^+ \cong R[e,e^-].$$
For a given filtered $\AI$-algebra, $(C,\{m_k\})$, by considering modulo
 $\NOVO^+$, we obtain
$$\OL{m}_k : B_k (\OL{C}) \otimes_R R[e,e^-] \to
C[1] \otimes_R R[e,e^-].$$
We assume that all the $\OL{m}_k$ maps in fact are induced from
$$\OL{m}_k: B_k (\OL{C}) \to \OL{C}[1].$$
We make similar assumptions for all unfiltered $\AI$-homomorphisms,
unfiltered $\AI$-bimodules and their homomorphisms in this paper as in \cite{FOOO}.

\begin{remark}
We clarify our notation of $\AI$-algebra. What we call $\AI$-algebra
here is called in some literature {\it weak} $\AI$-algebra which may have
a non-trivial $m_0$ term. A case without $m_0$ term is called a strict
$\AI$-algebra. 
\end{remark}

Considering the cyclic or symmetric group action, one can repeat the construction
of the previous subsection for filtered $\AI$-algebras.
\begin{definition}
Let $B^{cyc}_kC$ be the set of fixed elements of the cyclic group action on $B_kC$, and denote 
$$ \HH{B}^{cyc}C = \HH{\oplus}_{k=0}^\infty B^{cyc}_kC, \;\;\HH{B}^{cyc}_{\geq 1}C = \HH{\oplus}_{k=1}^\infty B^{cyc}_kC$$
We define $E_kC)$ to be the submodule of $B_k(C)$ consisting of
{\it fixed} elements of symmetric group action on $B_k(C)$ and
let $$\HH{E}C = \HH{\oplus}_{k=0}^\infty E_k(C), \;\;\HH{E}_{\geq 1}C = \HH{\oplus}_{k=1}^\infty E_k(C)$$
\end{definition}

We recall the notion of filtered $\AI$-homomorphism between two filtered $\AI$-algebras.
The family of maps of degree 0
$$f_k : B_k(C_1) \to C_2[1] \;\;\textrm{for} \; k =0,1,\cdots $$
induce the coalgebra map
$\HH{f}:\HH{B}C_1 \to \HH{B}C_2$, which for $x_1 \otimes  \cdots \otimes x_k \in B_k C_1$ is defined
by the formula
$$\HH{f}(x_1\otimes \cdots \otimes x_k) = \sum_{0 \leq k_1 \leq \cdots \leq k_n \leq k}
f_{k_1}(x_1,\cdots,x_{k_1})\otimes \cdots \otimes f_{k-k_n }(x_{k_n+1},\cdots,x_{k}).$$
We remark that the above can be an infinite sum due to the possible existence of $f_0(1)$.
In particular, $\HH{f}(1) = e^{f_0(1)}$. It is assumed that
\begin{equation}\label{fenergy}
\begin{cases}
f_k(F^\lambda B_k(C_1)) \subset F^\lambda C_2[1], \;\;\textrm{and}\\
f_0(1) \in F^{\lambda'}C_2[1] \,\;\; \textrm{for some}\; \lambda' >0. 
\end{cases}
\end{equation}
The map $\HH{f}$ is called a filtered $\AI$-homomorphism if  
$$\HH{d} \circ \HH{f} = \HH{f} \circ \HH{d}.$$
It is easy to check the following, whose proof is left as an exercise.
\begin{lemma}\label{fe}
For any filtered $\AI$-homomorphism $f:C_1 \to C_2$, 
the map $\HH{f}$ descends to the chain maps
$$\HH{f}: \HH{B}^{cyc}C_1 \to \HH{B}^{cyc}C_2, \;\;\; \HH{f}:\HH{E}C_1 \to \HH{E}C_2$$
\end{lemma} 
In particular the latter provides the notion of a filtered $\LI$-homomorphism between
filtered $\LI$-algebras.

\subsection{Filtered bimodules}
The notion of a filtered $\AI$ or $\LI$-bimodule can be easily defined as in unfiltered case.
Let $M$ be a graded free filtered $\NOVO$-module and denote by $F^{\lambda}M$ its
filtration. We complete $M$ with respect to this filtration.
Let $(C_1,\{m^1_k\}), (C_0,\{m^0_k\})$ be filtered $\AI$-algebras over $\NOVO$.

A family of operations for $k_1,k_0 \in \ZZ_{\geq 0}$
$$n_{k_1,k_0}:B_{k_1}(C_1) \otimes M[1] \otimes B_{k_0} (C_0) \to M[1]$$
of degree one define an $\AI$-bimodule structure on $M$ if they
satisfy the equation (\ref{ddmod}). We
also assume that $n_{*,*}$ preserves the filtration in an obvious way.

These operations can be extended to
\begin{equation}\label{bid}
\HH{d}:\HH{B}(C_1) \HH{\otimes} M[1] \HH{\otimes} \HH{B}(C_0) \to
\HH{B}(C_1]) \HH{\otimes} M[1] \HH{\otimes} \HH{B}(C_0).
\end{equation}
defined by (codifferentials of $\HH{B}C_i$ are denoted as $\HH{d}^i$)
$$ \HH{d}(x_1 \otimes \cdots \otimes x_k \otimes y \otimes z_1 \otimes \cdots \otimes z_l)
 = \HH{d}^1 ( x_1 \otimes \cdots \otimes x_k )\otimes y \otimes z_1 \otimes \cdots \otimes z_l$$
 $$ +  \sum_{p \leq k, q \leq l}
(-1)^{|x_1|' + \cdots + |x_{k-p}|'} x_1 \otimes \cdots \otimes x_{k-p} \otimes n_{p,q}(x_{k-p+1} \otimes \cdots \otimes y \otimes \cdots \otimes z_q) \otimes \cdots  \otimes z_l
$$
\begin{equation}\label{bimodulebar}
+ (-1)^{\sum(|x_i|' + |y|' )} x_1 \otimes \cdots \otimes x_k \otimes y \otimes
\HH{d}^0 (z_1 \otimes \cdots \otimes z_l).
\end{equation}
The family of maps $\{n_{k_1,k_0}\}_{k_1,k_0 \in \ZZ_{\geq 0}}$ defines an $\AI$-bimodule if 
\begin{equation}\label{ddmod}
\HH{d}\circ \HH{d}=0.
\end{equation}
One can rewrite the above equation into countably many equations involving $n_{k_1,k_0}$'s. The first equation
is
\begin{equation}\label{bimoduleeq}
n_{0,0} \circ n_{0,0} ( a ) +  n_{1,0}(m_0^{1}(1),a) + (-1)^{|a|'} n_{0,1}(a,m_{0}^{0}(1)) =0.
\end{equation}
In an unfiltered $\AI$-bimodule case, the equation (\ref{bimoduleeq}) becomes
$\OL{n}_{0,0} \circ \OL{n}_{0,0}=0$.

Now we recall the notion of an $\AI$-bimodule homomorphism.
Let $C_i$, $C_i'$ be filtered $\AI$-algebras ($i=0,1$). Let
$M$ and $M'$ be $(C_1,C_0)$ and $(C_1',C_0')$ filtered $\AI$-bimodules
respectively. Let $f^i: C_i \to C_i'$ be filtered $\AI$-algebra homomorphisms.
Then, a filtered $\AI$-bimodule homomorphism $\phi:M \to M'$ over $(f^1,f^0)$
is a family of $\NOVO$-module homomorphisms $\{\phi_{k_1,k_0}\}$

$$\phi_{k_1,k_0}:B_{k_1}(C_1) \HH{\otimes}_{\NOVO} M[1] \HH{\otimes}_{\NOVO} B_{k_0} (C_0) \to M'[1]$$
which respects the filtration in an obvious way, and satisfies
\begin{equation}
\HH{\phi} \circ \HH{d} = \HH{d}' \circ \HH{\phi}.
\end{equation}
Here $\HH{\phi}:B(C_1) \HH{\otimes} M[1] \HH{\otimes}
 B(C_0) \to
B(C_1') \HH{\otimes} M'[1] \HH{\otimes}
 B(C_0')$ is defined by
$$\HH{\phi}(x_1 \otimes \cdots \otimes x_k \otimes y \otimes z_1 \otimes \cdots \otimes z_l)
 = \sum_{p \leq k, q \leq l}
\HH{f}^1(x_1 \otimes \cdots \otimes x_{1,k-p}) \otimes $$
$$n_{p,q}(x_{k-p+1} \otimes \cdots \otimes y \otimes \cdots \otimes z_q)
\otimes \HH{f}^0(z_{q+1} \otimes \cdots  \otimes \cdots z_l).$$

In the case that $C_1$ and $C_0$ (resp. $C_1'$ and $C_0'$) are the same $\AI$-algebra, we denote $\WT{C}$(resp. $\WT{C}'$) the induced 
$\LI$-algebras from $C_i$ (resp. $C_i'$) for $i=1$ or $2$.
By taking the fixed elements of symmetric group action in the above construction, one can define a notion of 
filtered $\LI$-morphism $\phi : M \to M'$ over $f$.

Now, we recall the notion of a pullback of an $\AI$-bimodule.(See \cite{FOOO} Lemma 26.7 - 9.)
Let $(M,n )$ be a filtered $(C_1',C_0')$ $\AI$-bimodule, and let
$f^i:C_i \to C_i'$ (i=0,1) be filtered $\AI$-homomorphisms.
Then $(M,n )$ give rise to a $(C_1,C_0)$ $\AI$-bimodule
$\big( (f^1,f^0)^*M, (f^1,f^0)^*n \big)$ with
$$(f^1,f^0)^*n (\vec{x},y,\vec{z}) = n \big( \widehat{f}^1(\vec{x}),y,\widehat{f}^0(\vec{z}) \big)$$

The pull-back operation is also functorial.
Namely, let $g^i: C_i' \to C_i''$ be filtered $\AI$-homomorphisms
and $M'$ a filtered   $(C_1'',C_0'')$ $\AI$-bimodule.
Then, A filtered $\AI$-bimodule homomorphism $\phi : M \to M'$ over $(g^1,g^0)$
induces a filtered $\AI$-bimodule homomorphism over the identity
$$(f^1,f^0)^*\phi : (f^1,f^0)^*M \to (g^1 \circ f^1,g^0 \circ f^0)^*M',$$
where
$$(f^1,f^0)^*\phi (\vec{x},y,\vec{z}) = \phi \big( \widehat{f}^1(\vec{x}),y,\widehat{f}^0(\vec{z}) \big).$$
For the case $(f^1,f^0) = (id,id)$, it states that
an $\AI$-bimodule homomorphism $\phi:M \to M'$ over $(g^1,g^0)$ can be considered as
an $\AI$-bimodule homomorphism $\widetilde{\phi} : M \to (g^1,g^0)^*M'$ over $(id,id)$.

\subsection{Gapped condition and spectral sequences}
We recall the gapped condition and the spectral sequence arising
from the related energy filtration. Let $G$ be a submonoid of $\RR_{\geq 0} \times 2\ZZ$ satisfying the
following conditions
\begin{enumerate}
\item Let $\pi:\RR_{\geq 0} \times 2\ZZ \to \RR_{\geq 0}$ be the
projection to the first component. Then $\pi(G) \subset \RR_{\geq 0}$ is
discrete.
\item $G \cap (\{ 0 \} \times 2\ZZ) = \{(0,0)\}$
\item $G \cap (\{\lambda\} \times 2\ZZ)$ is finite set for any $\lambda$.
\end{enumerate}
We may denote its components as $\lambda,\mu$ :  For $\beta \in G$, 
$$\beta = (\lambda(\beta),\mu(\beta)) \in \RR_{\geq 0} \times 2\ZZ. $$
A filtered $\AI$-algebra is called to be {\it G-gapped }, if there
exist $R$-module homomorphisms $m_{k,\beta}:B_k \OL{C}[1] \to \OL{C}[1]$
for $k=0,1,2,\cdots$ and $\beta \in G$ such that
$$m_k = \sum_{\beta \in G} T^{\lambda(\beta)}e^{\mu(\beta)} m_{k,\beta}.$$
A filtered $\AI$-algebra $(C,m)$ is said to be gapped if it is $G$-gapped for some G.
Similarly one can define gapped $\AI$-homomorphisms.
We remark that the $\AI$-algebra of Lagrangian submanifolds
constructed in \cite{FOOO} is gapped due to Gromov Compactness theorem. Here 
$G$ is defined to be the submonoid of $\RR_{\geq 0} \times 2\ZZ$ generated by
$$G(L)_{0}=\{(\omega(\beta),\mu_L(\beta))|\beta \in \pi_2(M,L), \CM(L,\beta,J) \neq 0\},$$
If an $\AI$-algebra is gapped, then for any $k \geq 0$,
$$m_k ({\bf x}) - \OL{m}_k ({\bf x}) \in F^{\lambda_0} C
\;\;\textrm{for some}\;\; \lambda_0>0 $$
since $0$ is discrete in $\pi(G)$.

One can define in a similar way, gapped filtered $\AI$-bimodules,
gapped filtered $\AI$-bimodule homomorphisms. 
For weakly filtered $\AI$-bimodule homomorphism, \cite{FOOO} introduces a notion of a $G-set$, $G'$
and $G'$-gapped weakly filtered $\AI$-bimodule homomorphism. This is analogous to the above definition
but to allow energy loss up to a fixed amount. We refer readers to \cite{FOOO} Definition 21.3 for
details.
 
Let $(C,\delta)$ be a chain complex over $\NOVO$, which is gapped.
A new energy filtration is introduced by setting
$ \CF^n C = F^{n\lambda_0} C$ for each $n \in \ZZ_{\geq 0}$.
This filtration give rise to the spectral sequence. This is a
spectral sequence of a filtration over a filtered ring,
and we recall it here from \cite{FOOO}.
We put
\begin{equation}
\begin{cases}
Z_r^{p,q}(C) =\{ x \in \CF^q C^p| \delta(x) \in \CF^{q+r-1}C^{p+1} \}
+ \CF^{q+1}C^p, \\
B_r^{p,q}(C) =\big( \delta( \CF^{q-r+2}C^{p-1}) \cap \CF^q C^p \big)
+ \CF^{q+1}C^p, \\
\CE_r^{p,q}(C) = \frac{ Z_r^{p,q}(C)}{B_r^{p,q}(C)}.
\end{cases}
\end{equation}
We denote $\Lambda^{(0)}_{0,nov}$ to be the degree zero part
of $\NOVO$. We define a filtration on $\NOVO^{(0)}$ by
$\CF^n \NOVO^{(0)} = F^{n\lambda_0}
\NOVO^{(0)}$.
We denote
$\Lambda^{(0)}(\lambda) = \NOVO^{(0)}/F^{\lambda}\NOVO^{(0)}.$
Then the associated graded module is given by
$gr_*(\CF \NOVO^{(0)}) = \oplus_{n \in \ZZ_{\geq 0}} gr_n(\CF
\NOVO^{(0)}),$
where each $gr_n(\CF \NOVO^{(0)})$ is naturally isomorphic to
$\Lambda^{(0)}(\lambda)$.
Each $\CE_r^{p,q}$ has a structure of $\Lambda^{(0)}(\lambda)$-module.
\begin{lemma}[\cite{FOOO}Lemma26.20]
There exists $d_r^{p,q}: \CE_r^{p,q} \to \CE_{r}^{p+1,q+r-1}$, which  is a
$\Lambda^{(0)}(\lambda)$-module homomorphism such that
\begin{enumerate}
\item $\delta_r^{p+1,q+r-1} \circ \delta_r^{p,q} = 0$.
\item $Ker(\delta_r^{p,q})/Im(\delta_r^{p-1,q-r+1}) \cong \CE_{r+1}^{p,q}(C).$
\item $e^{\pm 1} \circ \delta_r^{p,q} = \delta_r^{p\pm 2, q} \circ e^{\pm 1}.$
\end{enumerate}
\end{lemma}
The convergence of this spectral sequence is a non-trivial question
since the filtration is not bounded and $\delta_r$ does not vanish for large $r$ in general.
In the case of the Floer cohomology (with respect to $m_1$), convergence of
the spectral sequence was proved in \cite{FOOO}.

\subsection{$\AI$ (and $\LI$)-homotopies}
We recall the notions of $\AI$-homotopies between two $\AI$-homomorphisms and
between two $\AI$-bimodule homomorphisms.
In \cite{FOOO}, it is defined using the notion of a model of
$[0,1] \times C$.

A filtered $\AI$-algebra $\FC$ together with filtered $\AI$-homomorphisms
$$ Incl: C \to \FC,\;\; Eval_{s=0}: \FC \to C, \;\; Eval_{s=1}:\FC \to C$$
is said to be a {\it model} of $[0,1] \times C$ if the following holds.
\begin{enumerate}
\item $Incl_k : B_k C \to \FC$ is zero unless $k=1$. The same holds for
$Eval_{s=0}$ and $Eval_{s=1}$.
\item $Eval_{s=0} \circ Incl =  Eval_{s=1} \circ Incl = identity$
\item $Incl_1$ induces a cochain homotopy equivalence of the complex
$(\OL{C},\OL{m}) \to (\OL{\FC},\OL{m})$, and $(Eval_{s=0})_1$,$(Eval_{s=1})_1$
induce cochain homotopy equivalences of the complex
$(\OL{\FC},\OL{m}) \to (\OL{C},\OL{m})$.
\item The homomorphism  $(Eval_{s=0})_1 \oplus (Eval_{s=1})_1: \FC \to C\oplus C$
is surjective.
\end{enumerate}

Let $C_1,C_2$ be filtered $\AI$-algebras and $f,g:C_1 \to C_2$ filtered $\AI$-homomorphisms between them.
Then $f$ is said to be homotopic to $g$ if there exists a filtered $\AI$-homomorphism
$\FF:C_1 \to \FC_2$ such that $Eval_{s=0} \circ \FF = f, \;\; Eval_{s=1} \circ \FF = g.$
\begin{equation}
\xymatrix{ \hspace{4cm}& C_2 & \hspace{1cm} \\
C_1  \ar[ur]^f \ar[r]^{\CF} \ar[dr]^g & \FC_2 \ar[u]_{Eval_0} \ar[d]^{Eval_1} & \hspace{1cm}\\
\hspace{4cm}& C_2 & \hspace{1cm}}
\end{equation}
Here $\FC_2$ is a model of $[0,1] \times C_2$, and the above definition is independent of the
choice of a model.

Now, we recall the model for $\AI$-bimodules.
Let $M$ be a filtered $(C_1,C_0)$ $\AI$-bimodule and $\FC_i$ a model
of $[0,1] \times C_i$. A model of $[0,1] \times M$ is a filtered
$(\FC_1,\FC_0) \; \AI$-bimodule $\FM$ equipped with
$\AI$-bimodule homomorphisms $Eval_{s=s_0}:\FM \to M$ over
$Eval_{s=s_0}:\FC_i \to C_i$ (for $s_0=0,1$), and $Incl:M \to \FM$ over $Incl:C_i \to \FC_i$ with
the following properties.
\begin{enumerate}
\item $Eval_{s=s_0} \circ Incl$ is equal to the identity
\item $(Eval_{s=s_0})_{k_1,k_0} = (Incl)_{k_1,k_0} = 0$ for $(k_1,k_0) \neq (0,0)$.
\item $(Eval_{s=0})_{0,0} \oplus (Eval_{s=1})_{0,0}: \FM \to M \oplus M $ is split surjective
\item $(Incl)_{0,0}:M \to \FM$ induces a cochain homotopy equivalence between $\OL{n}_{0,0}$
complexes.
\end{enumerate}

There also exists a notion of $\LI$-homotopy which is defined in an analogous way, and
we refer readers to \cite{Fu3} for explicit statements on them.

\subsection{Weakly filtered bimodule homomorphisms.}\label{sec:bhomotopy}
For a filtered $\AI$-bimodule $(M,n)$ over $\NOVO$, we get a
filtered $\AI$-bimodule $(\widetilde{M},n)$ over $\NOV$ by
$$\widetilde{M} = M \otimes_{\NOVO} \NOV.$$
Note that $\widetilde{M}$ has a filtration $F^{\lambda}\widetilde{M}$
over $\lambda \in \RR$.

Let $\WT{M}$ be a filtered $(C_1,C_0)$ $\AI$-bimodule over $\NOV$,
and $\WT{M}'$  be a filtered $(C_1',C_0')$ $\AI$-bimodule over $\NOV$.
Let $f^{(i)}:C_i \to C_i'$ be a filtered $\AI$-homomorphisms.
A {\it weakly filtered $\AI$-bimodule homomorphism} $\WT{M} \to \WT{M}'$
over $(f^{(0)},f^{(1)})$ is a family of $\NOV$-module homomorphisms
$$\phi_{k_1,k_0}:B_{k_1}(C_1) \HH{\otimes} \WT{M} \HH{\otimes} B_{k_0} (C_0) \to \WT{M}'$$
with the following properties:

\begin{enumerate}
\item There exists $c\geq 0$ independent of $k_0,k_1$ such that
$$\phi_{k_1,k_0} \big( F^{\lambda_1}B_{k_1}(C_1) \HH{\otimes} F^{\lambda}\WT{M} \HH{\otimes} F^{\lambda_0}B_{k_0} (C_0) \big) \subset F^{\lambda_1 +\lambda + \lambda_0 - c}\WT{M}'$$
\item  $\HH{\phi} \circ \HH{d} = \HH{d}' \circ \HH{\phi}$
\end{enumerate}

Weakly filtered homomorphisms arise when we study the
invariance property of the Floer cohomology $HF(L_0,L_1) \cong HF(L_0,\phi(L_1))$ where the constant $c$ is
related to the Hofer norm of the Hamiltonian isotopy $\phi$.

Let $\phi,\psi:M \to M'$ be (weakly) filtered $\AI$-bimodule homomorphisms
over $(f^{(1)},f^{(0)})$ and $(g^{(1)},g^{(0)})$ respectively.
Here $f^{(i)}:C_i \to C_i'$, $g^{(i)}:C_i \to C_i'$ are filtered
$\AI$-homomorphisms. Then, $\phi$ is said to be homotopic to $\psi$ if
there exists models $\FM', \FC_i'$ of $[0,1]\times M', [0,1]\times C_i'$ respectively,
homotopies $\FF^{(i)} :C_i \to \FC_i'$ between $f^{(i)}$ and $g^{(i)}$, and
a (weakly) filtered $\AI$-bimodule homomorphism
$\Phi :M \to \FM'$ over $(\FF^{(1)},\FF^{(0)})$ such that
$Eval_{s=0} \circ \Phi = \phi, \;\; Eval_{s=1} \circ \Phi = \psi:$

\begin{equation}
\xymatrix{ \hspace{4cm}& M' & \hspace{1cm} \\
M \ar[ur]^{\phi} \ar[r]^{\Phi} \ar[dr]^{\psi} & \FM' \ar[u]_{Eval_0} \ar[d]^{Eval_1} & \hspace{1cm}\\
\hspace{4cm}& M' & \hspace{1cm}}
\end{equation}

Here we also recall the definitions of homotopy equivalences.
A filtered $\AI$-homomorphism $f:C \to C'$ is called a {\it homotopy equivalence } if
there exists a filtered $\AI$-homomorphism $g :C' \to C$ such that $f \circ g$ and $g\circ f$ are
homotopic to identity.

An (weakly) filtered $\AI$-bimodule homomorphism $\phi:M \to M'$ over
$(f^0,f^1)$  is said to be a {\it homotopy equivalence}
if there exist an (weakly) filtered $\AI$-bimodule homomorphism $\psi:M' \to M$ over
$(g^0,g^1)$ where $\phi \circ \psi$ and $\psi \circ \phi$ are homotopic to identity.
Here $g^1$ and $g^0$  are homotopy inverses of $f^1$ and $f^0$ respectively.

The notions discussed so far can be carried out in the $\LI$-setting also.

\section{Bar cohomology and isomorphisms}
Consider the bar complexes $(B\OL{C},\WH{d})$ and $(\HH{B}C,\HH{d})$ and
its subcomplexes $(\HH{B}^{cyc}C,\HH{d})$ and $(\HH{E}C,\HH{d})$. In this section,
we first show that cohomology of the bar complex is trivial for weakly unital $\AI$-algebras. But the cohomology of subcomplexes
are not trivial in general. In fact, 
cohomology of $(\HH{B}^{cyc}C,\HH{d})$ is isomorphic to cyclic homology of $\AI$-algebra $C$
and the cohomology of $(\HH{E}C,\HH{d})$ is the cyclic Chevalley Eilenberg homology of the induced $\LI$-algebra $\WT{C}$ of the
given $\AI$-algebra $C$.
We prove isomorphisms of these cohomology theories under quasi-isomorphisms.

\subsection{Bar complex}
Let $(\OL{C},\OL{m})$ be a unital strict $\AI$-algebra.
The following theorem is well-known.
\begin{lemma}\label{bartrivial}
The cohomology of the bar complex $(B\OL{C},\WH{d})$ is trivial.
\end{lemma}
\begin{proof}
Let $I$ be the unit of the $\AI$-algebra. One can define the contracting homotopy $s$ of the
bar complex as follows:
$s:B\OL{C}  \to B\OL{C}$ is defined as
\begin{equation}
s(x_1 \otimes \cdots \otimes x_n) = I \otimes x_1  \otimes \cdots \otimes x_n.
\end{equation}
One can check without much difficulty that on $B\OL{C}$
$$ \WH{d} \circ s + s \circ \WH{d} = id - 0,$$
which provides the contracting homotopy of the bar complex.
\end{proof}
In the filtered case, we have
\begin{lemma}\label{bartrivial2}
The cohomology of the bar complex $(\HH{B}C,\HH{d})$ of a filtered homotopy unital $\AI$-algebra $C$ is isomorphic to $\NOVO$.
\end{lemma}
\begin{proof}
In the next subsection, we will show that one quasi-isomorphic $\AI$-algebras have isomorphic cohomologies of
the bar complexes. Hence, we may assume that the filtered $\AI$-algebra is unital, by taking the 
the canonical model of the given filtered $\AI$-algebra which is homotopy unital.

In the case with $m_0=0$, the same homotopy $s$ defined as in the unfiltered case provides
the contracting homotopy. The only difference in this case is that $\HH{B}C$ has in addition $B_0C=\NOVO$.
As for $1 \in \NOVO$, we have $\HH{d}(1) = m_0(1) =0$, and $1$ is a $\HH{d}$-cycle. But clearly, the image of
$\HH{d}$ never contains $1$ as one of its components. Hence $1$ generates $\HH{d}$-cohomology in this case and
this proves the lemma for $m_0=0$.

Let us assume that $m_0 \neq 0$. Unfortunately in this case, $s$ does not define a contracting homotopy
(see Lemma \ref{lem:contho}), hence this case is a bit more complicated.
As $\HH{d}(1) = m_0(1) \neq 0$, it is not clear whether the cohomology of the bar complex is trivial or isomorphic to $\NOVO$.
But we claim that it is always isomorphic to $\NOVO$. Namely, we can always find a $\HH{d}$-cycle $\gamma$ which includes $1$ as
one of its component and as before this gives rise to a non-trivial $\HH{d}$-cohomology element.
We define $\gamma$ as
$$\gamma := 1 - L \otimes m_0 + L \otimes m_0 \otimes L \otimes m_0 + \cdots = \sum_{k=0}^\infty (-1)^k (L \otimes m_0)^{\otimes k}.$$
Note that the sum is well-defined as the energy of the summand goes to infinity as $k \to \infty$.
We claim that $\HH{d}(\gamma) = 0$. This follows from the following which uses unitality of the $\AI$-algebra.
$$\HH{d}((L \otimes m_0)^{\otimes k}) = m_0 \otimes (L \otimes m_0)^{\otimes k}
- m_0 \otimes (L \otimes m_0)^{\otimes k-1}.$$
Now, to prove that $\HH{d}$-cohomology is isomorphic to $\NOVO$, we can proceed  as we do in the next subsection that we consider energy filtration and
use the vanishing of bar cohomology in the unfiltered case and spectral sequence arguments to prove the vanishing of $\HH{d}$-cohomology for
tensors of positive length.

Also note that in the special case that $\AI$-algebra is unobstructed and has a bounding cochain $b \in C$ satisfying $\HH{d}(e^b)=0$,  
$e^b$ can be used instead of $\gamma$. One can show that any two such $\HH{d}$-cycle containing $1$ is cohomologous from the vanishing results.
\end{proof}

\subsection{Isomorphisms}
Let us call the cohomologies of the complex $(B\OL{C},\WH{d})$, $(\HH{B}C,\HH{d})$, $(\HH{B}^{cyc}C,\HH{d})$ and $(\HH{E}C,\HH{d})$ as
bar cohomology for short. In this section, we prove that two quasi-isomorphic (filtered) $\AI$-algebras have
isomorphic bar cohomology. 

For the cyclic case, it can be also proved by showing its equivalence to the cyclic homology of $\AI$-algebra, but we show the proof here as the similar arguments are used at several instances of this paper. Exactly the same argument works for all cases, so we present the proof in the symmetric case only.

We first consider the unfiltered case (in particular, we have $m_0=0$).
\begin{prop}\label{unfilteredhomo}
Let $\OL{C}_1,\OL{C}_2$ be unfiltered $\AI$-algebras over a ring $R$, and let $\OL{f}:\OL{C}_1 \to \OL{C}_2$ be an unfiltered $\AI$-homomorphism
which induces an isomorphism on $\OM_1$-cohomologies. Then, $f$ induces an isomorphism on bar cohomology.
\end{prop}
\begin{proof}
Let $\WH{f}:B\OL{C}_1 \to B\OL{C}_2$ be the associated cohomomorphism between two coalgebras.
By the lemma \ref{fe}, we can regard it as a chain map
$$\WH{f}:(E\OL{C}_1,\WH{d}_1) \to (E\OL{C}_2,\WH{d}_2).$$

To prove that $\WH{f}$ induces an isomorphism on bar cohomology, we will
use the spectral sequences induced by the following number filtrations
on  $E\OL{C}_i$'s. Namely,
We set
\begin{equation}\label{numberf}
N^k(E\OL{C}_i) = E\OL{C}_i \cap B_{0,\cdots,k}\OL{C}_1.
\end{equation}
Note that
$$ 0 = N^0(E\OL{C}_i) \subset N^1(E\OL{C}_i) \subset \cdots \subset E\OL{C}_i. $$
And the filtration is exhaustive, and Hausdorff:
$$\cup_{k} N^k(E\OL{C}_i) =E\OL{C}_i,\;\; \cap_{k} N^k(E\OL{C}_i) = \emptyset.$$
(We remark that for a filtered $\AI$-algebra, the number filtration
is not exhaustive because of the completion with respect to the energy, even in the case that $m_0 = 0$).

It is also easy to check that $\WH{d}$ preserves the filtration.
Therefore for each $i$, there exists a spectral sequence
with
$$\CE_1^{p,q}(E\OL{C}_i) = N^q(E\OL{C}_i^{p}) /
N^{q-1}(E\OL{C}_i^{p}),$$
which converges to the homology of $\WH{d}$ on $E\OL{C}_i$ (See \cite{Mc},\cite{W} for example).
Convergence can be easily seen as the filtration is bounded below,
 exhaustive and Hausdorff (\cite{W} Theorem 5.5.1).
Here the differential $\delta_1$ on $\CE_1$ is induced by
$\WH{m}_1$. Hence,
$$\CE_2^{p,q}(E\OL{C}_i) = E\big( H^{p}(\OL{C}_i) \big)
\cap B_q H^{p}(\OL{C}_i).$$
Note that $\WH{f}$ induces a map of spectral sequences since $\WH{f}$ preserves the number filtration:
$$\WH{f}\big(N^k(E\OL{C}_1)\big) \subset N^k(E\OL{C}_2).$$
 Since $f_1$ induces an isomorphism on $\OL{m}_1$ cohomologies, it is easy to see that
$\WH{f}$ induces an isomorphism on the $\CE_2$ levels of the spectral sequences.
 Hence by the standard arguments of the spectral sequences, $\WH{f}$ induces an isomorphism
between $\WH{d}$ cohomologies of $E\OL{C}_i$.
The other cases follow from the same argument.
\end{proof}

Now, we consider a filtered case.
\begin{prop}\label{filteredhomo}
Let $C_1,C_2$ be gapped filtered $\AI$-algebras over $\NOVO$, and let $f:C_1 \to C_2$ be a gapped filtered $\AI$-homomorphism, where
$\OL{f}_1$ induces an isomorphism on $\OM_1$-homologies. Then, $f$ induces an isomorphism on bar cohomology
of $C_1$ and $C_2$.
\end{prop}
\begin{remark}
Note that the statement is over $\NOVO$ coefficients instead of $\NOV$. If $\WT{C}_1$ and $\WT{C}_2$ are
filtered $\AI$-algebras over $\NOV$ obtained from the above $C_1$ and $C_2$ by taking tensor products with $\NOV$, 
then $f:\WT{C}_1 \to \WT{C}_2$ also induces an isomorphism on bar cohomology.
\end{remark}  
\begin{proof}
First, note that $\HH{d}$ and $\HH{f}$ do not  preserve the number filtration. Namely,
$f_0(1),m_0(1) \in \Lambda^+_{0,nov}$ increases the number of tensor products in a given term.
We will consider the energy filtration and consider the associated spectral sequences.
Then, we prove that the induced map between spectral sequences on the $\CE_2$ level
is an isomorphism by considering a number filtration on $\CE_1$ level of the spectral sequences.

Recall that $\HH{E}C_i$ has an energy filtration, which we denoted as $F^{\lambda}(\HH{E}C)$.
By the gapped condition on $C_1,C_2$ and $f$, we can take $\lambda_0 >0$ which works for both of them. We consider
the filtration $\CF^n(\HH{E}C_i) = F^{n\lambda_0}(\HH{E}C_i)$.
It is easy to check that this filtration is complete:
$$ \HH{E}C_i = \lim_{\leftarrow} \HH{E}C_i / \CF^n (\HH{E}C_i).$$
Note that by (\ref{menergy}) and (\ref{fenergy})
$$\HH{d} \big( \CF^\lambda (\HH{E}C_i) \big) \subset \CF^\lambda (\HH{E}C_i), \;\;
\HH{f} \big( \CF^\lambda (\HH{E}C_1) \big) \subset \CF^\lambda (\HH{E}C_2).$$
Hence for each $i$, we have a spectral sequence with
$$\CE_1^{p,q} (\HH{E}C_i) = \CF^{q}(\HH{E}C_i^p) / \CF^{q+1}(\HH{E}C_i^p),$$
and a morphism of spectral sequences induced from $\HH{f}$. But the convergence
of these spectral sequences is not clear as $d_r \neq 0$ is even for large $r$ in general.

But we need the spectral sequences for comparison purposes only and for such a
purpose, convergence of the spectral sequences are not required by
the following general theorem on spectral sequences.

\begin{theorem}[\cite{W} Eilenberg-Moore Comparison Theorem 5.5.11]
Let $f:V \to W$ be a map of filtered complexes of modules, where both $V$ and $W$ are
complete and exhaustive. Fix $r \geq 0$. Suppose $f_r: \CE^{p,q}_r(V) \cong \CE^{p,q}_r(W)$ is an isomorphism
for all $p$ and $q$. Then $f:H^*(V) \to H^*(W)$ is an isomorphism.
\end{theorem}
The idea of the proof of the above theorem is to use the mapping cone complex, which is
also filtered by $F^q cone(f) = F^{q+r} V [1] \oplus F^q W$. And
the fact that $f^r$ is an isomorphism of $\CE_r$, implies that $\CE^r_{p,q}(cone(f))=0$ for all $p,q$ by the
related long exact sequence. In this case, spectral sequence obviously collapses and
one can apply the complete convergence theorem (see \cite{W})
to conclude that $H_*cone(f)$ is trivial. Since $cone(f)$ is an exact complex, this implies the above theorem.

In our case, note that we have
$$\CE_1^{p,q} \cong E\OL{C}^p \otimes_R gr_*( \mathcal{F} \NOVO).$$
And the differential $\delta_1$ on $\CE_1$-level is induced from $\WH{d}$ which is the energy zero part of
$\HH{d}$.

Now, we show that the induced map $f_*$ between $\CE_2$ levels of the spectral sequences is an isomorphism.
Note that  the induced map from $f$ between $\CE_1$ levels of the spectral sequences is
induced by $\WH{f}$, which is the energy zero part of $\HH{f}$.
Note also that $\OL{f}_0(1) = \OL{m}_0 = 0$ and elements of $E\OL{C}$ are of finite
sum since there cannot be an infinite sum without
having the energy going to infinity.

  Hence, for each fixed $p$ and $q$, we consider the number filtration
$N^k(E_1^{p,q}(E\OL{C}_i))$ as in (\ref{numberf}) and there exists another spectral sequence arising from this number filtration converging to the homology of $(\CE_1^{p,q}, \WH{d})$.
Note that $\WH{f}$ induces an isomorphism between $\OL{m}_1$-cohomologies. Hence, the
$\WH{f}$ induces an isomorphism of the spectral sequences from the number filtration,
and induces an isomorphism between homologies of $(\CE_1^{p,q}, \WH{d})$.

This shows that the induced map $f_*$ on the $\CE_2$-levels of the spectral sequences (with respect to the energy filtration) is indeed an isomorphism.
Hence by the Eilenberg-Moore comparison theorem, $\HH{f}$ induces an
isomorphism on bar cohomology. The other cases follow  exactly from the same argument. 
\end{proof}
\begin{remark}
In \cite{FOOO}, the spectral sequence of $(C,m_1)$ (not $\HH{B}C$ nor $\HH{E}C$)
with respect to the energy filtration was shown to converge, by
using the fact that $\AI$-algebras of Lagrangian submanifolds are weakly finite.
\end{remark}

\section{Hochschild, Chevalley-Eilenberg homology and isomorphisms}
In this section, we recall a definition of Hochschild (resp. Chevalley-Eilenberg) homology of an $\AI$ (resp. $\LI$)-bimodule and
consider their isomorphism properties under weakly filtered gapped homotopy equivalences.  Weakly filtered case where the related map is not filtration preserving, is essential to discuss the invariance of Lagrangian Floer homology when only one of the Lagrangian submanifold is
moved by an Hamiltonian isotopy.

\subsection{Definition of Hochschild homology}
We recall the definition of Hochschild homology of an $\AI$-bimodule $(M, \{ n_{*,*} \})$ of an $\AI$-algebra $A=(C,\{m_*\})$.(See \cite{GJ},\cite{KS} or \cite{S} for more details on this subsection).

We begin with a remark that the Hochschild homology can be regarded as a bar cohomology in the following way.
In \ref{def:modbar} , we have a complex $(B^M(C,C), D^M)$. In fact, one can consider a cyclic
group action and an induced subcomplex $(B^M(C,C)^{cyc},D^M)$ considering the fixed elements of such a cyclic action. The homology of this subcomplex is
called a Hochschild homology of a bimodule $M$ over $A$.

Now, we give more detailed and conventional description of Hochschild homology.
We denote
$$C^k(A,M) = M[1] \otimes C[1]^{\otimes k}.$$
We will denote its degree $\bullet$ part as $C^k_\bullet(A,M)$.

We define the Hochschild chain complex
\begin{equation}\label{def:hochchain}
C_{\bullet}(A,M) = \HH{\oplus}_{k \geq 0} C^k_\bullet(A,M),
\end{equation}
after completion with respect to energy filtration and 
with the boundary operation $$d^{Hoch} :C_{\bullet}(A,M) \to C_{\bullet+1}(A,M)$$  defined as follows: 
we will underline the module element for reader's convenience.
For $v \in M$ and $x_i \in A$, 
$$d^{Hoch} (\UL{v} \otimes x_1 \otimes \cdots \otimes x_k)
= \sum_{\stackrel{1 \leq j \leq k+1 -i}{1 \leq i}} (-1)^{\E_1} \UL{v} \otimes \cdots \otimes x_{i-1} \otimes m_j(x_i,\cdots,x_{i+j-1}) \otimes \cdots \otimes x_k$$
$$+ \sum_{i=1}^{k+1}  (-1)^{\E_2} \UL{v} \otimes x_1 \otimes \cdots \otimes x_{i-1} \otimes m_0(1) \otimes x_{i} \otimes \cdots \otimes x_k$$ 
\begin{equation}\label{eq2}
 + \sum_{\stackrel{0 \leq i, j \leq k}{ i+j \leq k}} (-1)^{\E_3} \UL{n_{i,j}\big(x_{k-i+1},\cdots,x_k,v, x_1,\cdots,x_{j} \big)} \otimes x_{j+1} \otimes \cdots 
\otimes x_{k-i}
\end{equation}

Here the sign $\E_1,\E_2,\E_3$ is obtained from Koszul sign convention as usual. More explicitly, we have 
$$\E_1 = \E_2 = |v|' + |x_1|' + \cdots + |x_{i-1}|',$$
$$\E_3 = \big(\sum_{s=1}^i |x_{k-i+s}|'\big)\big( |v|' + \sum_{t=1}^j |x_{t}|'\big).$$

The second and third type expressions in (\ref{eq2}) arise as in the figure.
One considers an element $v \otimes x_1 \otimes \cdots \otimes x_k$ as placed in a circle with special marking on the module element $v \in M$. And the boundary operation $d^{Hoch}$ may be understood as taking an appropriate operation on elements placed on a connected arc of the circle or the insertion of $m_0$, and reading off the resulting element starting from the special marking. 
\begin{figure}
\begin{center}
\includegraphics[height=1in]{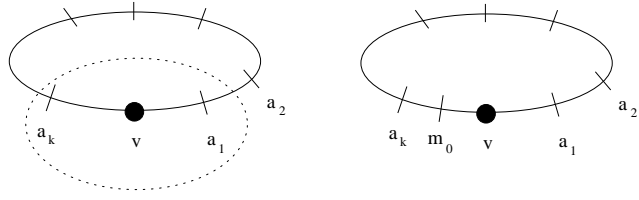}
\caption{Hochschild boundary}
\label{dhat}
\end{center}
\end{figure}
In particular, in the second terms of (\ref{eq2}), we do not insert $m_0$
ahead of $v$, because in the operation corresponding to the right hand side figure, $m_0$ will be inserted
in the last position, after $x_k$. 

The following  is standard and can be easily understood from the figure.
\begin{lemma}
$$d^{Hoch} \circ d^{Hoch}=0$$
\end{lemma}

The homology of $d^{Hoch}$ is called the {\em Hochschild homology } of $M$ over $A$ and is denoted as
$H_{\bullet}(A,M)$.
In the case of $M=A$, where the $\AI$-bimodule structure on $A$ is given by 
$$n_{i,j} = m_{i+j+1},$$
we have the Hochschild homology $H_{\bullet}(A,A) = HH_{\bullet}(A)$ of an $\AI$-algebra $A$.

\subsection{Definition of Chevalley-Eilenberg homology}\label{sec:CEho}
Here, we recall the definition of Chevalley-Eilenberg ($\LI$-algebra)  homology with coefficient in a $\LI$-bimodule.

Let $M$ be a $\LI$-bimodule over an $\LI$-algebra $\WT{A}=(C,\{l_*\})$ (see Definition \ref{def:lmod}).
We denote
$$CE^k(\WT{A},M) = M[1] \otimes E_kC$$
and denote its degree $\bullet$ part as $CE^k_\bullet(\WT{A},M)$.

We define the chain complex
$$CE_{\bullet}(\WT{A},M) = \HH{\oplus}_{k \geq 0} CE^k_\bullet(\WT{A},M),$$
after completion with respect to energy filtration.

The boundary operation $$d^{CE} :CE_{\bullet}(\WT{A},M) \to CE_{\bullet+1}(\WT{A},M)$$  defined as follows, using
the $\LI$-module structure maps $\eta$:
for $v \in M$ and $x_i \in C$ for $i=1,\cdots,k$, 
we define
$$d^{CE}(\UL{v} \otimes [x_1,\cdots,x_k]) = $$
$$\sum_{k_1 \geq 0} \sum_{(k_1,k-k_1) shuffle}  (-1)^{\E(\sigma,\vec{x}_{\geq 1})} \UL{\eta_{k_1}\big( \UL{v} \otimes [
x_{\sigma(1)},\cdots,x_{\sigma(k_1)}]\big)} \otimes [x_{\sigma(k_1+1)},\cdots,x_{\sigma(k)}] $$
$$+ \sum_{k_1 \geq 0} \sum_{(k_1,k-k_1) shuffle} (-1)^{|v|'+\E(\sigma,\vec{x}_{\geq 1})} \UL{v} \otimes 
[l_{k_1}([x_{\sigma(1)},\cdots,x_{\sigma(k_1)}]),x_{\sigma(k_1+1)},\cdots,x_{\sigma(k)}].$$
By the definition of $\LI$-bimodule, we have
 $$d^{CE} \circ d^{CE} =0,$$
 and we denote its homology as $H^{CE}_{\bullet}(\WT{A},M)$.
\subsection{Weakly filtered homotopy equivalences and isomorphisms}
We show isomorphism properties of the Hochschild or Chevalley-Eilenberg homology under weakly filtered homotopy equivalences.
To do so, we first show that there is a canonical chain map between the corresponding chain complexes.

Suppose we have an $\AI$-bimodule homomorphism $\phi:M \to N$ between two $\AI$-bimodules over $A$. Namely, we have
a family of maps $\phi_{i,j}: A^{\otimes i} \otimes M \otimes A^{\otimes j} \to N$ satisfying $\AI$-bimodule equations.

We define a chain map
$\phi_* : C_{\bullet}(A,M) \mapsto  C_{\bullet}(A,N)$ as $$\phi_*(v \otimes x_1 \otimes \cdots \otimes  x_k) = $$
$$\sum_{ \stackrel{0 \leq i, j \leq k}{i+j \leq k}} (-1)^{\E_2} \phi_{i,j}\big(x_{k-i+1},\cdots,x_k,v, x_1,\cdots,x_{j} \big) \otimes x_{j+1} \otimes \cdots 
\otimes x_{k-i},$$
where $\E_2$ is as given above.
One can check without much difficulty that 
\begin{lemma} we have
$$\phi_* \circ d^{Hoch}_M =   d^{Hoch}_N \circ \phi_*,$$
and hence it induces a map $\phi_* : H_\bullet(A,M) \to H_\bullet(A,N).$
\end{lemma}

More generally, let $A=(C,m)$, $A'=(C,m')$ be two filtered $\AI$-algebras, and let $\alpha:A \to A'$ be a 
filtered $\AI$-homomorphism. Let $M$ (resp. $M'$) be
an $\AI$-bimodule over $A$ (resp. over $A'$). For an $\AI$-bimodule homomorphism $\phi:M \to M'$ over $(\alpha,\alpha)$,  one can define a chain map
$\phi_*: C_{\bullet}(A,M) \to C_{\bullet}(A',M'):$
$$\phi_*(v \otimes x_1 \otimes \cdots \otimes  x_k) = $$
$$\sum_{\stackrel{0 \leq i, j \leq k}{i+j \leq k}} (-1)^{\E_2} \phi_{i,j}\big(x_{k-i+1},\cdots,x_k,v, x_1,\cdots,x_{j} \big) \otimes \HH{\alpha}(x_{j+1} \otimes \cdots 
\otimes x_{k-i}).$$

One can also obtain similar maps for $\LI$-case. Namely, let $\WT{A}$(resp. $\WT{A}'$) be an induced $\LI$-algebra from $A$(resp. $A'$), and
$\WT{f}$ an induced filtered $\LI$-homomorphism between $A$ and $A'$.  And consider
$M$(resp. $N$) as an induced $\LI$-module over $\WT{A}$(resp. $\WT{A}'$). The $\AI$-bimodule map $\phi$ induces an $\LI$-bimodule homomorphism between $M$ and $N$ over $\WT{f}$. One can also check that such a map is a chain map.

\begin{prop}\label{homotopyequiv}
Let $M$ and $N$ be  gapped filtered $\AI$-bimodules over an  gapped filtered $\AI$-algebra $A = (C,m)$. Let $\phi:M \to N$ be filtered or weakly filtered  gapped $\AI$-bimodule homomorphism, which is a homotopy equivalence.
Then the map $\phi$ induces an isomorphism between Hochschild homology of $M$ and $N$ over $A$ and
also $\phi$ induces an isomorphism between Chevalley-Eilenberg homology of $M$ and $N$ over $\WT{A}$.
$$ H_{\bullet}(A,M) \cong H_{\bullet}(A,N),\;\; H^{CE}_{\bullet}(\WT{A},M) \cong H^{CE}_{\bullet}(\WT{A},N)$$
\end{prop} 
\begin{proof}
The proof follows by considering the definition of homotopy via the models of $\AI$ or $\LI$-homotopy.
As the proof of $\LI$-case is exactly the same as that of $\AI$-case, we only consider $\AI$-case. 

From the definition of homotopy equivalence and the $\AI$-homotopy, it is enough to prove that
if $f,g:M \to N$ are filtered or weakly filtered $\AI$-bimodule homomorphisms over $A$, and if they
are $\AI$-homotopic to each other, then 
they induce the same map on Hochschild homology (i.e. $f_* = g_*$).

Denote by  $\FN$ the model of $[0,1] \times N$, and by $\FA$ the model of $[0,1] \times A$.
We will use the notation $Incl, Eval$ without distinction between two models $\FN$ and $\FA$, which should be clear from the context.

We denote by $H$ the $\AI$-homotopy  $H: M \to \FN$ which is a (weakly) filtered $\AI$-bimodule homomorphism over $(Incl, Incl)$
($\FN$ is an $\AI$-bimodule over $\FA$).
It satisfies the following commutative diagram:
\begin{equation}
\xymatrix{ \hspace{4cm}& N & \hspace{1cm} \\
M \ar[ur]^{f} \ar[r]^{H} \ar[dr]^{g} & \FN \ar[u]_{Eval_0} \ar[d]^{Eval_1} & \hspace{1cm}\\
\hspace{4cm}& N & \hspace{1cm}}
\end{equation}
Note that $Eval \circ Incl =id$ by the definition of the model, hence, the composition
$Eval_s \circ H: M \to N$ is an $\AI$-bimodule map over $(id,id)$ for $s=0,1$.

Since the induced maps are $f_* = (Eval_0)_* \circ H_*$ and  $g_* = (Eval_1)_* \circ H_*$, it is enough to 
show the composition $Eval_s \circ Incl$ induces an isomorphism on Hochschild homology as it implies that
$$g_* = (Eval_1)_* \circ H_*=  (Eval_0 \circ Incl)_* \circ (Eval_1)_* \circ H_* = 
(Eval_0)_* \circ  (Incl \circ Eval_1)_* \circ H_* = f_*.$$

Even though $H$ is only weakly filtered, the maps $Eval$, $Incl$ are filtered, and are of very simple forms
that $(Eval_s)_{i,j} =0$ and $(Incl_s)_{i,j}=0$ for $(i,j) \neq (0,0)$ and $s=0$ or $1$. We can use these good properties of $Eval$ and $Incl$ maps to prove the desired isomorphism property. Now the rest of the proof is very similar to that of the last section and we leave the details to the reader.
\end{proof} 
\subsection{Reduced Hochschild homology with $m_0$ terms}
It is well-known that for a unital $\AI$-algebra $A$ with $m_0 =0$, the Hochschild homology of a $\AI$-bimodule $M$ over $A$ can be computed using the reduced Hochschild chain complex. Similarly one considers in the filtered case, 
$$C_\bullet^{red}(A,M) = \HH{\oplus}_k 
M[1] \otimes \big(C/ (k \cdot I) \big)[1]^{\otimes k},$$
and it is easy to check that 
$$\HH{d}:C_\bullet^{red}(A,M) \to C_{\bullet+1}^{red}(A,M)$$ is well-defined and defines a complex whose homology is
called the reduced Hochschild homology $H_\bullet^{red}(A,M)$.

We prove that with the Novikov field coefficients, it is quasi-isomorphic to the standard Hochschild chain complex.
Instead of $\NOV$, we use the coefficients $\Lambda$ or $\NOVE$ which are Novikov fields defined
in the definition (\ref{def:lambda}) the (\ref{def:nove}). We consider $\NOVE$ only for simplicity.
We give a proof since the standard proof does not
generalize immediately due to the presence of $m_0$.
\begin{prop}\label{reducedhh}
Reduced Hochschild homology is isomorphic to the Hochschild homology
$H_\bullet(A,M)$ for filtered $\AI$-algebra $A$ with a strict unit $I$ (even with $m_0 \neq 0$).
\end{prop}
\begin{proof}
We modify the proof given in the book of Loday\cite{L} section 1.6.
Define $s_i: C_\bullet(A,M) \to C_{\bullet-1}(A,M)$ by
$$s_i(v \otimes a_1 \otimes \cdots \otimes a_k) = (-1)^{|v|' +  \cdots + |a_i|'}  v \otimes \cdots \otimes a_i \otimes I \otimes a_{i+1}
\otimes \cdots \otimes a_k,$$
and define $t_i: C_\bullet(A,M) \to C_{\bullet}(A,M)$ by
$$t_i(v \otimes a_1 \otimes \cdots \otimes a_k) = (-1)^{|v|'} v \otimes \HH{m}_0 ( a_1 \otimes \cdots \otimes a_i \otimes I \otimes a_{i+1}
\otimes \cdots \otimes a_k).$$
Here, we set $s_i = t_i =0$ if $i<0$ or $i > k$.

The maps $t_i$'s are introduced to make the filtration below compatible with the Hochschild differential.
Note that for short we may write 
\begin{equation}\label{eq:ti}
t_i= d^{Hoch}_0 \circ s_i,
\end{equation}
where $d^{Hoch}_0$ denotes the 2nd term of the definition of $d^{Hoch}$ in (\ref{eq2}).
Let $D_\bullet$ be a submodule of $C_\bullet(A,M)$  which is the completion of the submodule generated by the images of  the maps $\{s_i \}_{i\in \NN \cup \{0\}}, \{t_i\}_{i\in \NN \cup \{0\}}$ or
equivalently by the images of 
$\{s_i \}_{i\in \NN \cup \{0\}}$.

One  check easily that $(D_\bullet,d^{Hoch})$ is a subcomplex of
$(C_\bullet(A,M),d^{Hoch})$. 

\begin{lemma}
$(D_\bullet,d^{Hoch})$ is acyclic subcomplex.
\end{lemma}
\begin{proof}
We prove this by introducing the following filtration:
Consider a filtration  $F_pD_\bullet$ which is generated by the  
images of  $s_0,\cdots,s_p,t_0,\cdots,t_p$. One can check that the filtration is compatible with $d^{Hoch}$.

Then, by the spectral
sequence argument, it is enough to show that $Gr_p D_\bullet$ is acyclic for any $p$, which will be shown in
the next lemma.
Here, even though $F_pD_\bullet$ is not exhaustive filtration as we have used completion, but the filtration is complete.
And in this case, acyclicity for each $p$ implies that the spectral sequence is weakly convergent and hence
proves the acyclic property of the subcomplex $(D_\bullet,d^{Hoch})$(see \cite{W}).
\end{proof}

\begin{lemma}
$Gr_p D_\bullet$ is acyclic for any $p$. More precisely, we have a chain homotopy $\alpha_p$ 
between identity and zero map: i.e they satisfies the identities
$$(d^{Hoch} \circ \alpha_p + \alpha_p \circ d^{Hoch} ) \circ s_p = s_p \;\;\; \textrm{mod}\;\; F_{p-1},$$
$$(d^{Hoch} \circ \alpha_p + \alpha_p \circ d^{Hoch} ) \circ t_p = t_p \;\;\; \textrm{mod}\;\; F_{p-1}$$
\end{lemma}
\begin{proof}
We define a chain homotopy $\alpha_p:Gr_p D_\bullet \to Gr_p D_{\bullet-1}$ as follows.
$$\alpha_p (s_p ( v \otimes a_1 \otimes \cdots \otimes a_k)) = s_p \circ s_p ( v \otimes a_1 \otimes \cdots \otimes a_k),$$
$$\alpha_p (t_p(v \otimes a_1 \otimes \cdots \otimes a_k)) = 
-t_p(s_p(v \otimes a_1 \otimes \cdots \otimes a_k)).$$
From (\ref{eq:ti}), we may also
write 
$$\alpha_p \circ t_p = -  d^{Hoch}_0 \circ s_p \circ s_p.$$

We write $$d^{Hoch} = d^{Hoch}_0 + d^{Hoch}_+.$$
One can check as in the standard case (although complicated)
$$(d^{Hoch}_+ \circ \alpha_p + \alpha_p \circ d^{Hoch}_+ ) \circ s_p = s_p \;\;\; \textrm{mod}\;\; F_{p-1},$$
$$(d^{Hoch}_+ \circ \alpha_p + \alpha_p \circ d^{Hoch}_+ ) \circ t_p = t_p \;\;\; \textrm{mod}\;\; F_{p-1}.$$

Now, we check the same identity for $d_0^{Hoch}$.
Note that 
$$(d^{Hoch}_0 \circ \alpha_p + \alpha_p \circ d^{Hoch}_0 ) \circ s_p =  d^{Hoch}_0 \circ s_p \circ s_p - d^{Hoch}_0 \circ s_p  \circ s_p =0.$$
$$(d^{Hoch}_0 \circ \alpha_p + \alpha_p \circ d^{Hoch}_0 ) \circ t_p =  -d^{Hoch}_0 \circ d^{Hoch}_0 \circ s_p \circ s_p + \alpha_p ( d^{Hoch}_0 \circ  d^{Hoch}_0 \circ s_p )=0,$$
as $d^{Hoch}_0 \circ  d^{Hoch}_0 =0$.
This proves the lemma.

\end{proof}
Hence $D_\bullet$ is acyclic, and the quotient complex is $C_\bullet^{red}(A,M)$.
\end{proof}

We remark that this reduced version should be helpful for computations of Hochschild homology for weakly obstructed Lagrangian submanifolds
as one can then ignore Hochschild boundary operation comming from $m_0$ by using the reduced version.

\section{Cyclic and cyclic Chevalley-Eilenberg homology}
We first define cyclic homology of $\AI$-algebra and cyclic Chevalley-Eilenberg homology of an induced $\LI$-algebra. Moreover, as there
are several approaches to define cyclic homology, we show how the standard approaches work out with $m_0 \neq 0$. It turns out that there
are some modifications to be made due to the presence of $m_0$. We refer readers to to the book by Loday \cite{L} in
the standard case of associative and Lie algebras and to the paper by Hamilton and Lazarev \cite{HL}
for an approach using non-commutative de Rham theory for these homology theories in the case $m_0=0$, and 

Let $A=(C,\{m_*\})$ be a filtered $\AI$-algebra, and let $\WT{A}$ be an induced $\LI$-algebra.
First, consider the subcomplexes of the bar complex $(\HH{B}^{cyc}C,\HH{d})$ and $(\HH{E}C,\HH{d})$
introduced in section 2. 
\begin{definition}\label{def:cycho}
We define the {\em cyclic homology} of an $\AI$-algebra  to be the homology of the complex $(\HH{B}^{cyc}_{\geq 1}C,\HH{d})$,
and denote it as $HC_\bullet(A)$.
We define the {\em cyclic Chevalley-Eilenberg homology} of an $\LI$-algebra $\WT{A}$ to be 
the homology of the complex $(\HH{E}_{\geq 1}C,\HH{d})$ and denote it as $HC^{CE}_\bullet(\WT{A})$.
\end{definition}
\begin{remark}
One may define cyclic homology using (\ref{eq:quo}) instead
as the usual definition for cyclic homology is given by considering the quotient as in (\ref{eq:quo}) of Connes' complex. But in our
case, the resulting homologies are isomorphic. 
\end{remark}

In fact  cyclic Chevalley-Eilenberg homology for $\LI$-algebra is just a Chevalley-Eilenberg homology with trivial coefficient. The reason
that it is called cyclic Chevalley-Eilenberg homology is that  there is a uniform approach for $\AI$, $\LI$ and $C_\infty$-algebras
to define Hochschild and cyclic homology theories via considering formal manifolds and their non-commutative de Rham theory. We refer readers to \cite{HL} for more detailed explanations and references.

\subsection{Cyclic bicomplex and Connes complex}
We show that the standard connections between several approach to define cyclic homology holds true with $m_0 \neq 0$.
The bicomplex for cyclic homology was introduced by B. Tsygan, and we can consider an analogous bi-complex for filtered $\AI$-algebras.
(See \cite{L} for the classical case and we assume that the reader is familiar with the construction in \cite{L}.)
Consider the Hochschild chain complex $C_\bullet(A,A)$ defined in (\ref{def:hochchain}).
For the cyclic generator $t_{n+1} \in \ZZ/(n+1)\ZZ$, we define its action on $A^{\otimes (n+1)}$  as in (\ref{gaction}):
$$t_{n+1} \cdot (x_0,x_1,\cdots,x_n) = (-1)^{|x_n|'(|x_0|'+\cdots+|x_{n-1}|')} (x_n,x_0,\cdots,x_{n-1}).$$
Here, we set $t_1$ to be identity on $A$ and write the identity map as $1$. Consider
$N_{n+1}:=1+t_{n+1}+t_{n+1}^2+ \cdots+ t_{n+1}^n$. 

As in the classical case,
we have the natural augmented exact sequence:
$$A^{\otimes (n+1)} \stackrel{1-t_{n+1}}{\longleftarrow} A^{\otimes (n+1)}  \stackrel{N_{n+1}}{\longleftarrow}  A^{\otimes (n+1)}  \stackrel{1-t_{n+1}}{\longleftarrow}  A^{\otimes (n+1)}  \stackrel{N_{n+1}}{\longleftarrow} \cdots .$$   
We consider $\oplus_{n=1}^\infty N_n$ action on $\oplus_{n=1}^\infty A^{\otimes n}$ and denote it as
\begin{equation}\label{symop}
N:C_\bullet(A,A) \mapsto C_\bullet(A,A).
\end{equation}
We can also similarly define $(1-t):C_\bullet(A,A) \mapsto C_\bullet(A,A)$.

Recall that in the classical case, cyclic bicomplex has even columns which are the copies of the Hochschild complex, and
odd columns which are the copies of the bar complex. We will construct the bicomplex in the similar way:
even columns will be given by $(C_\bullet(A,A),d^{Hoch})$. For odd columns, note that $\HH{B}C = \NOVO \oplus C_\bullet(A,A)$ as
$B_0C=\NOVO$ is not present in the Hochschild chains. Consider $\HH{d}$ operation on $C_\bullet(A,A)$ considered as
a subspace of $\HH{B}C$. Due to the lemma \ref{bartrivial2}, the homology of the chain complex $(C_\bullet(A,A),\HH{d})$ vanishes,
and this will be the odd columns.

These two differentials are certainly different.
For example, given any $x \in C$, we have $$d^{Hoch}(x) = (-1)^{|x|'}x \otimes m_0 + m_1(x) $$ whereas 
$$\HH{d}(x)=  m_0 \otimes x + (-1)^{|x|'}x \otimes m_0 + m_1(x).$$

To follow the standard notation, we set $b = d^{Hoch}$ and $b'=\HH{d}$. 
\begin{lemma}
We have on $C_\bullet(A,A)$ the following identities:
\begin{equation}
b (1-t) = (1-t) b', \;\;\; b'N=Nb .
\end{equation}
\end{lemma}

We thus obtain the cyclic bi-complex (analogous to Tsygan's) defined as follows.
\begin{definition}\label{def:cycbi}
Define $$CC_{pq}(A) = C_q(A,A) \;\;\textrm{for all} \; p\geq 0, q \in \ZZ.$$
We define differentials as
$$b:CC_{pq}(A) \mapsto CC_{p(q+1)}(A) \;\; \textrm{for} \; p \;\textrm{even}$$
$$-b':CC_{pq}(A) \mapsto CC_{p(q+1)}(A) \;\; \textrm{for} \; p \;\textrm{odd}$$ 
$$1-t:CC_{pq}(A) \mapsto CC_{(p-1)q}(A) \;\; \textrm{for} \; p \;\textrm{odd}$$
$$N:CC_{pq}(A) \mapsto CC_{(p-1)q}(A) \;\; \textrm{for} \; p \;\textrm{even}$$

\begin{equation*}
\xymatrix{ \hspace{1cm} & \hspace{1cm} & \hspace{1cm} & \hspace{1cm} & \hspace{1cm} \\
C_{1}(A,A) \ar[u]_{b}  &  C_{1}(A,A) \ar[u]_{-b'} \ar[l]_{1-t} & C_{1}(A,A) \ar[u]_{b} \ar[l]_{N} &  C_{1}(A,A) \ar[u]_{-b'} \ar[l]_{1-t} 
& \; \ar[l]_{N} \\
C_{0}(A,A) \ar[u]_{b}  &  C_{0}(A,A) \ar[u]_{-b'} \ar[l]_{1-t} & C_{0}(A,A) \ar[u]_{b} \ar[l]_{N}  
 & C_{0}(A,A) \ar[u]_{-b'} \ar[l]_{1-t} & \; \ar[l]_{N} \\
C_{-1}(A,A) \ar[u]_{b}  &  C_{-1}(A,A) \ar[u]_{-b'} \ar[l]_{1-t} & C_{-1}(A,A) \ar[u]_{b} \ar[l]_{N} 
 & C_{-1}(A,A) \ar[u]_{-b'} \ar[l]_{1-t} & \; \ar[l]_{N}  \\
\hspace{1cm} \ar[u]_{b}  &  \hspace{1cm} \ar[u]_{-b'} & \hspace{1cm} \ar[u]_{b} &  \hspace{1cm} \ar[u]_{-b'} &  \hspace{1cm} }
\end{equation*}
\end{definition}
For example, for associative algebras (whose degree is concentrated at zero), the standard bicomplex of cyclic homology can be seen in the $4$-th quadrant. All elements have degree $-1$ after degree shifting, hence the negative of the length gives the degree of an expression.

\begin{prop}\label{thm:cycbarcyc}
The homology of the above (completed) total complex $\HH{\textrm{Tot}}(CC(A))$ is isomorphic to
cyclic homology $HC_\bullet(A)$.
\end{prop}
\begin{proof}
In the standard case, there exists an isomorphism of the homology of the total complex of the bicomplex, and 
the homology of Connes' complex. 
Recall that Connes complex is defined as
\begin{equation}\label{eq:quo}
C^\lambda_\bullet(A):=coker(1-t) = C_\bullet(A,A) / im(1-t).
\end{equation}
It is easy to check that this complex has the same homology as cyclic homology defined in the definition \ref{def:cycho}
(which is the standard invariant and coinvariant relation).

Consider a natural surjection $p:\HH{Tot}(CC_\bullet)(A) \to C^\lambda_\bullet(A)$, where the quotient map is given from the first column. Recall that the rows of the bicomplex are acyclic augmented complexes with $H_0 = C^\lambda_\bullet(A)$.
Consider the standard horizontal increasing filtrations on $CC_\bullet(A)$ and $C^\lambda_\bullet(A)$, and use the spectral sequence arguments 
as in the classical case or as in  the last section to prove the proposition.
\end{proof}

As usual, there exist the Connes exact sequence, relating Hochschild homology and cyclic homology.
\begin{lemma}\label{conexact}
We have a following exact sequence.
$$ \to H_{\bullet}(A,A) \to HC_{\bullet}(A) \to HC_{\bullet +2}(A) \to H_{\bullet +1}(A,A) \to $$
\end{lemma}

\subsection{$(b,B)$-complex}
There is also the $(b,B)$-complex, which is obtained from the bicomplex using the acyclicity of the even columns.
We will see that the contraction homotopy of the bar complex in the standard case does not work for filtered $A_\infty$ algebras with $m_0\neq 0$.
We first find a modified contraction homotopy, and we will also discuss normalized $(b,B)$-complex.
For this, we need to work on Novikov ring which is a field. 
Instead of $\NOV$, we can use the coefficients $\Lambda$ or $\NOVE$ which are Novikov fields defined
in the definition (\ref{def:lambda}) the (\ref{def:nove}). We consider $\NOVE$ only for simplicity.
\begin{lemma}\label{lem:contho}
Consider $C_\bullet(A,A)$ with $\NOVE$ coefficient.
There exist a contracting homotopy $\tilde{s}$ of the complex $(C_\bullet(A,A), \HH{d})$ for a filtered $\AI$-algebra $A$.
\end{lemma}
\begin{proof}
If $m_0=0$, we have a contracting homotopy as in the standard case defined as
$s: C_\bullet(A,A) \to C_\bullet(A,A)$ defined by
$$s(x_1\otimes \cdots \otimes x_n) = I \otimes x_1 \otimes \cdots \otimes x_n.$$
It is easy to check that we have
$$s \circ \HH{d} + \HH{d}\circ s = id - 0.$$

But in the filtered case with non-trivial  $m_0$, one can easily notice that it is no longer true:
for any $x_1$, we have 
\begin{eqnarray*}
(s \circ \HH{d} + \HH{d}\circ s)(x_1) &=& I \otimes \big( m_1(x_1) + m_0 \otimes x_1 + (-1)^{|x_1|'}x_1 \otimes m_0\big) \\
 && + m_2(I,x_1) + m_1(I)\otimes x_1 - I \otimes m_1(x_1) + m_0 \otimes I \otimes x_1  \\
 &&- I \otimes m_0 \otimes x_1 -(-1)^{|x_1|'} I \otimes x_1 \otimes m_0 \\
 &=& x_1 + m_0 \otimes I \otimes x_1 \neq x_1.
\end{eqnarray*}
Hence, we will modify our contracting homotopy $s$ to $\tilde{s}$ in the following way.
First, recall from the lemma \ref{bartrivial2} that the homology of the bar complex $(\HH{B}C, \HH{d})$ is isomorphic to $\NOVO$.
It implies that the homology of $(C_\bullet(A,A),\HH{d})$ vanishes (with $\NOVE$ or $\NOVO$ coefficient) as we do not
consider the part $B_0C$.
\begin{lemma}
There exist a $\NOVE$ module $V$ such that 
$$C_\bullet(A,A) = Ker(\HH{d}) \oplus V.$$
\end{lemma}
\begin{proof}
We first note that $ Ker(\HH{d})  = Im (\HH{d})$ is $\NOVE$-module as $\HH{d}$ is $\NOVE$-linear.
As $\NOVE$ is a field, it is then easy to find such subspace $V$.
\end{proof}
Now, we define $\tilde{s}$ using the decomposition in the lemma. 
Define for $p \in V$, $\tilde{s}(p) = s(p)$ as before.
Define for $p \in Ker (\HH{d}) = Im (\HH{d})$, consider
$q \in V$ such that $p = \HH{d}(q)$ and define
$$ \tilde{s} (p) = I \otimes p - m_0 \otimes I \otimes q.$$

Then, $\tilde{s}$ now satisfies
\begin{equation}\label{stilde}
\tilde{s} \circ \HH{d} + \HH{d}\circ \tilde{s} = id - 0 .
\end{equation}
To see this, 
for $p \in V$, 
$$\tilde{s} \circ \HH{d}(p) + \HH{d}\circ \tilde{s}(p)
= I \otimes \HH{d}(p) - m_0 \otimes I \otimes p + \HH{d} ( I \otimes p)$$
For convenience, write $\HH{d} = \HH{d}_0 + \HH{d}_+$ where $\HH{d}_0 = \HH{m}_0$.
From the unital property of $\HH{d}$, we have 
$$\HH{d}(I \otimes p) = \HH{d}_0 (I \otimes p) + p - I\otimes \HH{d}_+(p).$$
Here, the second term is result of $\HH{d}_+$-operation containing $I$ in its input and the third term
is that of $\HH{d}$ which does not contain the unit $I$.
Hence the (\ref{stilde}) follows by adding up and computing $\HH{d}_0$.

For $p \in Ker(\HH{d})$ and $q \in V$ such that $p = \HH{d}(q)$,
we have
$$\tilde{s} \circ \HH{d}(p) + \HH{d}\circ \tilde{s}(p)
= \tilde{s} \circ (\HH{d}(\HH{d}(q))) + \HH{d} \circ \tilde{s}(\HH{d}(q))$$
\begin{equation}\label{tilde3}
= 0 + \HH{d}\big( I \otimes \HH{d}(q) - m_0 \otimes I \otimes q \big)
\end{equation}
Note that $\HH{d}( I \otimes \HH{d}(q))$ equals
\begin{equation}\label{tilde1}
\HH{d}_0(I \otimes \HH{d}_0 q) + \HH{d}_+(I \otimes \HH{d}_0q) + \HH{d}_0 (I \otimes \HH{d}_+q) +  \HH{d}_+ (I \otimes \HH{d}_+q )
\end{equation}
\begin{eqnarray}\label{tilde2}
\,&=& \big(m_0 \otimes I \otimes \HH{d}_0 q + I \otimes \HH{d}_0 (\HH{d}_0(q))\big) + \big(\HH{d}_0(q)) - I \otimes \HH{d}_+(\HH{d}_0(q)) \big) \nonumber \\
&&+ \big( m_0 \otimes I \otimes \HH{d}_+q - I \otimes \HH{d}_0(\HH{d}_+(q)) \big) 
+ \big(\HH{d}_+q - I \otimes \HH{d}_+(\HH{d}_+(q)) \big) \nonumber \\
&=& m_0 \otimes I \otimes \HH{d}q + \HH{d}(q) - I \otimes (\HH{d} \circ \HH{d}(q)) \nonumber \\
&=& m_0 \otimes I \otimes \HH{d}q + \HH{d}(q)
\end{eqnarray}
Here, we used the fact that $\HH{d}_0 \circ \HH{d}_0=0$.

Now, the third term of (\ref{tilde3}) equals
\begin{eqnarray}\label{tilde4}
-\HH{d}(m_0 \otimes I \otimes q) &=& -\HH{d}_0(m_0\otimes I \otimes q) - \HH{d}_+(m_0\otimes I \otimes q) \nonumber \\
&=& - m_0 \otimes I \otimes \HH{d}_0(q) -  m_0 \otimes I \otimes \HH{d}_+(q) \nonumber \\
&=& -m_0 \otimes I \otimes \HH{d}q
\end{eqnarray}
The claim follows by adding (\ref{tilde2}) with (\ref{tilde4}).
\end{proof}

This lemma is now used to define a new Connes operator $B=(1-t)\tilde{s}N$. As we have used $\tilde{s}$ instead of $s$, the resulting
$B$ is slightly different from the standard $B$ which may contain additional terms of the form $(1-t)(m_0 \otimes I \otimes \alpha)$.

And we obtain the following bicomplex whose homology is isomorphic to cyclic homology.
\begin{equation}
\xymatrix{ \hspace{1cm} & \hspace{1cm} & \hspace{1cm} \\
C_{1}(A,A) \ar[u]_{b}  &  C_{2}(A,A) \ar[u]_{b} \ar[l]_{B} & C_{3}(A,A) \ar[u]_{b} \ar[l]_{B} & \;    \ar[l]_{B}  \\
C_{0}(A,A) \ar[u]_{b}  &  C_{1}(A,A) \ar[u]_{b} \ar[l]_{B} & C_{2}(A,A) \ar[u]_{b} \ar[l]_{B}  & \;  \ar[l]_{B} \\
C_{-1}(A,A) \ar[u]_{b}  &  C_{0}(A,A) \ar[u]_{b} \ar[l]_{B} & C_{1}(A,A) \ar[u]_{b} \ar[l]_{B}  & \; \ar[l]_{B} \\
\hspace{1cm} \ar[u]_{b}  &  \hspace{1cm} \ar[u]_{b} & \hspace{1cm} \ar[u]_{b}  }
\end{equation}

Now, as we have $(b,B)$-complex, we can also consider normalized $(b,B)$-complex by considering
$C_{i}^{red}(A,A)$ instead of $C_{i}(A,A)$ in the $(b,B)$-complex above. There is an obvious surjection
from $(b,B)$-complex to normalized $(b,B)$-complex which can be shown to be quasi-isomorphism.

We remark that even though we have normalized $(b,B)$ complex, we do not have normalized Tsygan's bicomplex.
Also, the additional terms of Connes operator $B$ in the filtered case will disappear in the normalized $(b,B)$-complex, hence
giving rise to the standard $B$-operator. One can also define variants of cyclic homologies as in the standard case.

\section{Lagrangian Floer theory}
We recall some of the main results  of \cite{FOOO}, and apply to them the homology theories discussed so far.
Their invariance properties can be proved as a corollary.

\begin{theorem}[FOOO, Theorem A]
To each relatively spin Lagrangian submanifold $L$, we can associate
a structure of gapped filtered $\AI$-algebra structure $\{m_k\}$ on
$H^*(L,\NOVO)$, which is well-defined up to isomorphism. If
$\psi:(M,L) \to (M',L')$ is a symplectic diffeomorphism, then we can
associate to it an isomorphism $\psi_*:=(\psi^{-1})^*: H^*(L,\NOVO)
\to H^*(L',\NOVO)$ of filtered $\AI$-algebras whose homotopy class
depends only of the isotopy class of symplectic diffeomorphism
$\psi$.

 The Poincare dual $PD[L] \in H^{0}(L,\NOVO)$ of the fundamental class $[L]$ is
the unit of our filtered $\AI$-algebra.
\end{theorem}
They first construct a filtered $\AI$-algebra
$(C(L,\NOVO),m)$ which is homotopy unital, and use the following theorem of the canonical model construction,
to obtain a filtered $\AI$-algebra structure on homology $H^*(L,\NOVO)$, which is unital.

\begin{theorem}[FOOO, Theorem 23.2]
Any gapped filtered $\AI$-algebra $(C,m)$ is homotopy equivalent to a gapped
filtered $\AI$-algebra $(C',m')$ with $\OL{m}_1 =0$. The homotopy equivalence
can be taken as a gapped filtered $\AI$-homomorphism. If $(C,m)$ is homotopy unital, then
its canonical model is unital.
\end{theorem}

Recall that a filtered $\AI$-algebra is called canonical if $\,
\OL{m}_1 =0$. Note that the canonical model still may have 
non-trivial $m_0 \in \NOVO^+$ and $m_{1,\beta}$ for $\beta \neq 0$.

Now, we apply the homology theories discussed so far.
\begin{definition}
Let $(C(L),m)$ be a gapped filtered $\AI$-algebra of a Lagrangian submanifold $L$.
The Hochschild (resp. cyclic) homology of $(C(L),m)$ is called
{\bf Hochschild (resp. cyclic) Floer homology } of $L$ and denoted as
$$HH_\bullet(C(L),m) =:HH_\bullet(L)\;\; \big(resp. \;\; HC_\bullet(C(L),m) =: HC_\bullet(L) \big). $$
Let $(C(L),l)$ be the induced $\LI$-algebra from $(C(L),m)$.
The Chevalley-Eilenberg homology of $(C(L),l)$ with coefficient in $(C(L),l)$ (resp. $\NOV$) 
is called {\bf Chevalley-Eilenberg (resp. cyclic Chevalley-Eilenberg) Floer homology }of $L$ and denoted as
$$H^{CE}_\bullet((C(L),l),(C(L),l)) =:H^{CE}_\bullet(L,L)\;\; \big(resp. \;\; H_\bullet^{CE}((C(L),l),\NOV) =: HC_\bullet^{CE}(L) \big). $$
\end{definition}

As mentioned in the introduction, the main motivation to study these homology theories is that
they provide well-defined homology theories even when the original $\AI$-structure is obstructed, and
they are invariant under various choices involved.

\begin{corollary}
Hochschild, cyclic and (cyclic) Chevalley-Eilenberg Floer homologies are well-defined up to isomorphism depending only
on the homotopy class of the $\AI$-algebra of Lagrangian submanifold.
\end{corollary}
\begin{proof}
 Theorem A of \cite{FOOO}, with the
Proposition \ref{filteredhomo}, Proposition \ref{homotopyequiv} proves the corollary.
\end{proof}

To study the Lagrangian intersection theory, the case of a pair of Lagrangian submanifolds
$(L_1,L_0)$ is considered.
\begin{theorem}[FOOO, Theorem 12.72]
Let $L_1,L_0$ be a relatively spin pair of Lagrangian submanifolds, which
are of clean intersection. Then we have
$$\big(C(L_1,L_0),n \big)$$ which has the structure of  filtered $\AI$-bimodule over the pair
$$\big( \big(C(L_1,\NOVO),m_*\big), \big(C(L_0,\NOVO),m_*\big)).$$
\end{theorem}

Now, we restrict to the case when the Lagrangian submanifold $L_1$ is obtained as a Hamiltonian isotopy of $L_0 =L$
(namely, $L_1 = \phi_1(L)$ where $\phi_s, ( s\in [0,1])$ is a Hamiltonian isotopy with $\phi_0
=id$). By \cite{FOOO} Theorem 19.1, we have a homotopy equivalence
$$f:(C(L,\NOVO),m) \to (C(\phi_1(L),\NOVO),m).$$

By using $f$, we can pull-back the $\AI$-bimodule $\big(
C(\phi_1(L),L), n\big)$ to be an $\AI$-bimodule  $\big(
C(\phi_1(L),L), (f,id)^*n\big)$ over a pair $\big(
(C(L,\NOVO),m),(C(L,\NOVO),m)\big)$ as explained in section 2.4.

\begin{theorem}[FOOO, Theorem 12.75]
Let us assume $L=L_1=L_0$. We also assume $J_t$ is independent of $t$.
Then the $\AI$-bimodule structure on $C(L_1,L_0)$ can be taken as
the same as the $\AI$-algebra structure on $C(L,\NOVO)$.
\end{theorem}

The following theorem of \cite{FOOO}  proves the invariance of Floer cohomology.
\begin{theorem}[FOOO, Theorem 22.14]
There exists an $\epsilon$-weakly filtered $\AI$-bimodule homomorphism
$\Phi : (C(\phi_1(L),L;\NOV),n) \to (C(L,L;\NOV),n)$ over $(f,id)$, which
is a homotopy equivalence.
Here $\epsilon$ is any number greater than the Hofer length of the
Hamiltonian isotopy $\{\phi_s\}_s$.
\end{theorem}

Consider the $\AI$-algebra $A=(C(L;\NOV),m)$, which also can be regarded as $\AI$-bimodule $(C(L,\NOV),n)$ over $(A,A)$.
Denote by $M$  the $\AI$-bimodule 
$$\big( C(\phi_1(L),L;\NOV), (f,id)^*n \big)\;\; \textrm{over} \;(A,A).$$
Then, from the above theorem, together with the pull-back construction, we obtain the weakly filtered homotopy
equivalence $\Phi : M \to (C(L,\NOV),n)$ over $(id,id)$ between the two $\AI$-bimodules over $(A,A)$. 

Let us denote also by $\WT{A}$ the induced $\LI$-algebra from the $\AI$-algebra $A$.
and let $\WT{M}$ be the induced $\LI$-bimodule over $\WT{A}$ obtained from the $\AI$-bimodule $(M,(f,id)^*n)$ over $A$.

We emphasize that the following two theorems hold even for obstructed Lagrangian submanifolds.
\begin{corollary}\label{isoiso}
We have isomorphisms of Hochschild and Chevalley-Eilenberg homology:
\begin{eqnarray*}
H_{\bullet}(A,M) &\cong H_\bullet(A,A) &= HH_{\bullet}(L) \\
H_{\bullet}^{CE}(\WT{A},\WT{M}) &\cong H_{\bullet}^{CE}(\WT{A},\WT{A}) &= H^{CE}_{\bullet}(L,L) 
\end{eqnarray*}
\end{corollary}
\begin{proof}
This follows from the theorem 22.14 of \cite{FOOO} and the Proposition \ref{homotopyequiv}.
\end{proof}
\begin{corollary}\label{isoiso2}
If a Lagrangian submanifold $L$ is displaceable via Hamiltonian isotopy $\phi^1$ (i.e. $L \cap \phi^1(L) = \emptyset$),
then, its Hochschild Floer homology and Chevalley-Eilenberg Floer homology of $L$ vanish.
\end{corollary}
\begin{proof}
This directly follows from the above corollary as the module $M$ or $\WT{M}$ would be void in such a case.
\end{proof}
This proves the Theorem 1.3 stated in the introduction.

\section{Unobstructedness and Hochschild homology}
In this section, we discuss a relation between Maurer-Cartan elements and Hochschild homology of an $\AI$-algebra.
Namely, we prove 
\begin{prop}
Let $b$ be a Maurer-Cartan element of an unital $\AI$-algebra $A$.
Then, the following element $\gamma_b$ gives a Hochschild homology cycle of an $\AI$-algebra.
$$\gamma_b= I \otimes  e^b.$$
\end{prop}
We remark that the correspondence does not guarantee a non-vanishing Hochschild homology class. The reason is that
when the Lagrangian submanifold is unobstructed and displaceable, then its Hochschild homology  should vanish,
due to the corollary \ref{isoiso2}.

We briefly recall the definition of unobstructedness.
Consider a filtered $\AI$-algebra $A=(C,m)$ with  $\HH{d}(1) = m_0(1) \neq 0$. Then, we have $m_1^2 \neq 0$ in general.
Suppose there exist an element $b \in C^1$ which satisfies the following equation: 
$$\HH{d}{e^b} =\HH{d}(  1 + b + b \otimes b +  b \otimes b  \otimes b + \cdots) =0.$$
If such an element exist, the $\AI$-algebra $A$ is called {\em unobstructed}, and $b$ is called
a bounding cochain or Maurer-Cartan elements.

With any such $b$, and one can deform the $\AI$-algebra $(C,m)$ into another $\AI$-algebra $A^b=(C,m^b)$ by defining
the new $\AI$-structure as
$$m_k^b(x_1 \otimes x_2 \otimes \cdots \otimes x_k) := m(e^b \otimes x_1 \otimes e^b \otimes \cdots \otimes e^b  \otimes x_k \otimes e^b), $$
for $x_1 \otimes  \cdots \otimes x_k \in B_kC$.
Here $$m(e^b \otimes x_1 \otimes e^b \otimes x_k \otimes e^b) = \sum_* m_{k+*}(b,\cdots,b,x_1,b,\cdots,b,x_k,b,\cdots,b).$$

Note that if $b$ is a Maurer-Cartan element, we have
$$\HH{d}(e^b \otimes x_1  \otimes e^b) = \HH{d}(e^b) \otimes x_1 \otimes e^b + e^b \otimes m( e^b \otimes x_1 \otimes e^b) \otimes e^b
+ e^b \otimes x_1 \otimes \HH{d}(e^b) = e^b m_1^b(x_1) e^b.$$
This implies that
$$0 = \HH{d} \circ \HH{d} (e^b \otimes x_1  \otimes e^b) = \HH{d} (e^b \otimes m_1^b(x_1) \otimes  e^b) = 
e^b \otimes \big( (m_1^b)^2x_1 \big)
\otimes e^b.$$
Hence, $m_1^b$ defines a deformed chain complex whose homology in the Lagrangian case is called Lagrangian Floer homology. See \cite{FOOO} for more details.

Now, we begin the proof of the proposition.
\begin{proof}
As $b$ is a Maurer-Cartan element, we have $m(e^b)=0$.
Then, consider $\gamma_b$ defined as above, and it is easy to check that $d^{hoch}(\gamma_b)=0$.
We have 
\begin{eqnarray*}
d^{hoch}(\gamma_b) &=&m(I\otimes e^b) \otimes e^b - I \otimes e^b \otimes m(e^b) \otimes e^b + m(e^b \otimes I) \otimes e^b \\
&=& m_2(I,b) +m_2(b,I) = b + (-1)^{|b|}b = 0.
\end{eqnarray*}
\end{proof}

\begin{prop}
Let $A=(C,m)$ be unobstructed $\AI$-algebra. For any bounding cochain $b$ of $A$, consider a deformed $\AI$-algebra $A^b=(C,m^b)$.
Then, Hochschild homology of $A^b$ is independent of $b$ and for each $b$ we have
$$HH_\bullet(A) \cong HH_\bullet (A^b).$$
\end{prop}
\begin{proof}
To prove this, we only need to show that $(A,m^b)$ and $(A,m)$ are homotopy equivalent as in \cite{FOOO} Lemma 5.2.12. Such homotopy
equivalence $i^b:(C,m^b) \to (C,m)$ can be given by defining 
$$i^b_0(1) = b, \;\; i_1^b = id,\;\; i_{\geq 2}^b =0.$$
\end{proof}

\begin{lemma}
If two bounding cochains are gauge equivalent, then the induced Hochschild homology cycles  are
homologous
\end{lemma}
\begin{proof}
Let $\FA$ be the unital model of $[0,1] \times A$. By definition, two bounding cochains $b_0$ and $b_1$ are gauge equivalent,
if there exists a bounding cochain $\bf{b}$ of $\FA$ such that $Eval_{s}(\bf{b}) = b_s$ for $s=0$ and $1$.
As $\FA$ is unital $\AI$-algebra, let $\bf{I}$ be the unit of $\FA$. Then, $\bf{I}\otimes e^{\bf{b}}$ defines
a Hochschild cycle of $HH_\bullet(\FA)$. Also, note that $Eval_s$ induces a map between Hochschild cycles and
in fact as $(Eval_s)_k=0$ for $k \neq 0$, we have
$$ I \otimes e^{b_s} = Eval_s(\bf{I} \otimes e^{ \bf{b} })$$ 
But as $Eval_s$ and $Incl$ induces an isomorphism of Hochschild homology,
we may proceed as in the proof of Proposition 4.3 to prove that $I \otimes e^{b_s}$ has the same Hochschild homology 
for $s=0$ and $1$.
\end{proof}
We also remark that with a suitable Hochschild homology class(or in general a negative cyclic homology class), we can find explicit homotopy cyclic inner product structure on the $\AI$-algebra which will be explained in an upcoming joint work with Sangwook Lee
(see also \cite{C3}).

Now, we show that after dualization (see section \ref{sec:dual}), unobstructedness corresponds to the notion of an augmentation
(see for example \cite{Che}, \cite{EGH} for more details on augmentation) . Here an augmentation of a differential graded algebra $(B,d)$ is
an algebra homomorphism $\epsilon : B \to \kk$ to its coefficient ring $\kk$ such that $\epsilon \circ d =0$. The correspondence follows
easily from the formalism of \cite{FOOO}.
\begin{lemma}
Let $(A,m)$ be a filtered $\AI$-algebra over the Novikov field $\Lambda$. Suppose $(A,m)$ is unobstructed.
Then, the differential graded algebra $\big((\HH{B}A)^*,\HH{d}^* \big)$ has an augmentation.
\end{lemma}
\begin{proof}
Consider the $\AI$-automorphsim $i^b$ defined above, and also an induced map
$\HH{i}^b: \HH{B}A \to \HH{B}A$. Then, the corresponding augmentation $\epsilon:(\HH{B}A)^* \to \Lambda$ is
defined as a composition of the algebra map $(\HH{i}^b)^*: (\HH{B}A)^* \to (\HH{B}A)^*$ with the
projection $\pi_0:(\HH{B}A)^* \to \Lambda$ to its component of length zero:
$$\epsilon = \pi_0 \circ \HH{i}^b.$$
Hence it remains to show that $\epsilon \circ \HH{d}^* =0$.
Given $f \in (\HH{B}A)^*$, we have
$$\epsilon \circ \HH{d}^*(f) =  \pi_0 \circ \HH{i}^b \circ \HH{d}^*(f) = f(\HH{d}(\HH{i}^b(1))) = f(\HH{d}(e^b)) = 0.$$
\end{proof}

\section{Non-trivial element in cyclic Floer homology}
In this section, we find a condition of an obstructed case which has non-trivial cyclic Floer homology.
Let $L$ be a Lagrangian submanifold which only admit non-positive Maslov index pseudo-holomorphic discs. 
Namely, we assume that $\mu(\beta) \leq 0$ for any homotopy class $\beta$ which is realized by $J$-holomorphic discs.
Consider the unital $\AI$-algebra $A$ on $(H_*(L,\NOV),m)$, which is
given by the Theorem A of \cite{FOOO}. 

We assume that $A$ is obstructed. In an unobstructed case, the same result holds true with much easier proof using
the last part of the proof given here, and in this case $PD[L]$ gives a non-trivial element of cyclic Floer homology.
Hence we assume that $A$ is obstructed. Now we find a non-trivial element in $HC_\bullet(L)$. Denote by $m_0 = m_0(1) \neq 0$ and also recall that $PD[L]$ defines a unit on this gapped filtered $\AI$-algebra. To
simplify expression we will write $L$ instead of $PD[L]$.

Note that $L$ is not a cycle in the bar complex as we have
  $$\HH{d}(L) =
m_0 \otimes L - L \otimes m_0 \neq 0.$$ Our idea is to consider the following additional terms to cancel these $m_0$ terms successively. 
Recall the cyclic symmetrization operation $N$ from (\ref{symop}) and define
$$\alpha_{2k+1}= N_{2k+1} \big(\underbrace{L \otimes m_0 \otimes L \otimes m_0 \otimes  \cdots \otimes m_0 \otimes L}_{2k+1} \big)$$
We let $\alpha_{1} = L$ and consider the sum
$$\alpha = \sum_{k=0}^{\infty} (-1)^k \alpha_{2k+1} \in \HH{B}^{cyc}H(L,\NOVO)$$
\begin{prop}\label{propnonzero}
With the above assumptions, the element $\alpha$ defines a non-trivial homology class in $HC_{(-1)}(L)$.
\end{prop}
\begin{proof}
Note that in the expression of cyclic permutation of $\alpha$, any
two of $m_0$ are always separated by $L$. Because $L$ is a unit of
the $\AI$-algebra, the only non-trivial operations of $\HH{d}$ on $\alpha$ are
$\HH{m}_0$, $\HH{m}_1$ and $\HH{m}_2$. Since $m_1(L) = m_1(m_0) =0$,
we have $\HH{m}_1=0$. Therefore, it suffices to prove the following
lemma to prove the proposition.
\begin{lemma}\label{twozero} We have
$$\HH{m}_0(\alpha_{2k-1}) = \HH{m}_2(\alpha_{2k+1}).$$
\end{lemma}
\begin{proof}
We will compute both sides and show that they are indeed equal.
We first point out that both $m_0$ and $L$ have shifted degree one, hence when they pass across each other 
the negative sign will appear. We also have from (\ref{unit}) that
\begin{equation}\label{m0m2}
m_2(m_0,L) = m_2(L,m_0) = m_0.
\end{equation}

The left hand side can be computed by the following elementary lemma, whose proof is left for
the reader. 
\begin{lemma} Suppose $a_i$ for $i=1,\cdots,2k+1$ are elements of degree one. Then we have
$$\HH{m}_0 \big( N_{2k+1}(a_1 \otimes a_2 \otimes \cdots \otimes a_{2k+1}) \big)
= N_{2k+2}\big( \HH{m}_0(a_1 \otimes a_2 \otimes \cdots \otimes a_{2k-1})
\otimes a_{2k+1} \big).$$
\end{lemma}

Now, by using the lemma, we can compute 
\begin{eqnarray*}
\HH{m}_0 (\alpha_{2k-1}) &=& \HH{m}_0 \big( N_{2k-1}(\underbrace{L \otimes m_0 \otimes L \otimes \cdots \otimes m_0 \otimes L}_{2k-1}) \big) \\
&=& N_{2k} \big(\HH{m}_0 ( L \otimes m_0 \otimes L \otimes \cdots \otimes m_0)  \otimes L) \big) \\
&=& N_{2k} \big( \underbrace{m_0 \otimes L \otimes m_0 \otimes L \otimes \cdots \otimes m_0  \otimes L}_{2k} \big) 
\end{eqnarray*}
$$= k(m_0 \otimes L \otimes m_0 \otimes L \otimes \cdots \otimes m_0  \otimes L) -
k( L \otimes m_0 \otimes L \otimes \cdots \otimes m_0  \otimes L \otimes m_0).$$
The second line follows from the previous lemma, and the third line follows
from the cancellation ( the terms with $\cdots m_0 \otimes m_0 \cdots$ occur twice with the opposite signs).

Now we compute  $\HH{m}_2(\alpha_{2k+1})$.
Note that $\alpha_{2k+1}$ may be divided into the following 5 types from the cyclic permutations.
\begin{eqnarray*}
\alpha_{2k+1} &=& N_{2k+1}\big(\underbrace{L \otimes m_0 \otimes L \otimes \cdots \otimes m_0 \otimes L}_{2k+1} \big) \\
&=& L \otimes m_0 \otimes L \otimes \cdots \otimes m_0 \otimes L \\
&&+ L \otimes L \otimes m_0 \otimes  \cdots \otimes L \otimes m_0 \\
&&+ m_0 \otimes L \otimes \cdots \otimes  m_0 \otimes L \otimes L \\
&&+ m_0 \otimes L \otimes  \cdots \otimes m_0 \otimes L \otimes L \otimes m_0 \otimes \cdots \otimes m_0 \otimes L + \cdots \\
&&+ L \otimes m_0 \otimes  \cdots \otimes m_0 \otimes L \otimes L \otimes m_0 \otimes \cdots \otimes L \otimes m_0 + \cdots
\end{eqnarray*}
Note that the last two types have $(k-1)$ such elements each.
For each type, one can easily compute using (\ref{m0m2})
\begin{equation*}
\begin{array}{l}
\displaystyle \HH{m}_2 \big(L \otimes m_0 \otimes L \otimes \cdots \otimes m_0 \otimes L)
= m_0 \otimes L \otimes \cdots \otimes m_0 \otimes L - L \otimes m_0 \otimes L \otimes \cdots \otimes m_0, \\
\displaystyle \HH{m}_2 \big( L \otimes L \otimes m_0 \otimes  \cdots \otimes L \otimes m_0    \big) = 0, \\
\HH{m}_2 \big( m_0 \otimes L \otimes \cdots \otimes  m_0 \otimes L \otimes L \big) = 0, \\
\HH{m}_2 \big(  m_0 \otimes L \otimes  \cdots \otimes m_0 \otimes L \otimes L \otimes m_0 \otimes \cdots \otimes m_0 \otimes L \big) = m_0 \otimes L \otimes \cdots \otimes m_0 \otimes L, \\
\HH{m}_2 \big( L \otimes m_0 \otimes  \cdots \otimes m_0 \otimes L \otimes L \otimes m_0 \otimes \cdots \otimes L \otimes m_0 \big) =-L \otimes m_0 \otimes L \otimes \cdots \otimes m_0.
\end{array}
\end{equation*}

Hence we have
$$ \HH{m}_2(\alpha_{2k+1}) =
k(m_0 \otimes L \otimes m_0 \otimes L \otimes \cdots \otimes m_0  \otimes L)
-k( L \otimes m_0 \otimes L \otimes \cdots \otimes m_0  \otimes L \otimes m_0)
$$
Hence this proves the Lemma \ref{twozero}.
\end{proof}
So far we have proved that $\HH{d}(\alpha) =0$.
To prove the Proposition \ref{propnonzero}, we need to prove that $\alpha$ is a non-trivial element of the cyclic Floer homology.
We will need the assumption that Maslov index is non-positive for $J$-holomorphic discs for this purpose.

Recall that we have
$$m_k = \sum_{\beta \in G} T^{\lambda(\beta)}e^{\mu(\beta)/2} m_{k,\beta}.$$
Here $m_{k,\beta} : (H^*(L)[1])^{\otimes k} \to H^*(L)[1]$ has
(after degree shift) degree $1-\mu(\beta)$. And before degree shift, $m_{k,\beta}$ has degree $(2-\mu(\beta)-k)$.

Note that the available degrees of elements in $H^*(L)[1]$ are from $(-1)$ to $n-1$, and the only degree $(-1)$ element is $L$.
Also the shifted degree of $m_{k,\beta}(x_1,\cdots,x_k)$ is 
$$|x_1|' + |x_2|' + \cdots + |x_k|' + 1 - \mu(\beta).$$
Hence, if $x_i \neq L$ for all $i$, then $|x_i|' \geq 0$ hence
the $m_{k,\beta}(x_1,\cdots,x_k)$ has degree (before shift) $\geq 2 - \mu(\beta) > 0$. Hence they cannot produce $L$ as its image.

But as it is unital, if one of $x_i=L$, then most of the $m_k$ operations vanish (see (\ref{unit})) and the only non-trivial operation which can have $L$ as its image is
$m_2(L,L)$. But $L\otimes L$ is not an element of
$\HH{B}^{cyc}(H^*(L,\NOV))$, since
$$N_2(L\otimes L)= L\otimes L - L \otimes L = 0.$$
Hence, this proves that $\alpha$ is a non-trivial homology element in cyclic Floer homology, as
the leading term of $\alpha$ is not
in the image of $\HH{d}$.
This proves the proposition.
\end{proof}

We remark that similar approach in the cyclic Chevalley-Eilenberg complex does not work.
For example, one may check that the symmetric sum of the expression $L\otimes m \otimes L$ vanishes due to the
cancellation of pairs occurring in the permutation of two $L$'s.

The element $\alpha$ can be also seen as a cycle of the bicomplex given in the definition \ref{def:cycbi}.
To see this, note that $\alpha$ has degree $(-1)$ and satisfies $(1-t)\alpha=0.$
One should put $\alpha$ in the augmented bicomplex of the one given in the definition \ref{def:cycbi}.
Namely, consider $\alpha$ as an element in $C_{-1,-1}(A,A)$. As $(1-t)\alpha=0$, we can find $\alpha'_0$ with $N(\alpha')=\alpha$.
Also as $b'(\alpha)=0$, from the commutative diagram of the bicomplex, we
have $N(b(\alpha'_0))=0$, hence one can find $\alpha'_1$ with $(1-t)\alpha'_1 = b(\alpha'_0)$. One can continue in a similar way
to obtain a cycle in cyclic bicomplex. 

The original motivation for our interest in this non-trivial element was to prove the non-displaceability of 
Lagrangian submanifolds with Maslov class zero. To prove such a result, one may prove the non-vanishing of
the Hochschild homology of the $\AI$-algebra of such a Lagrangian submanifold.
Unfortunately, we do not know how to prove such a non-vanishing property of Hochschild homology of an $\AI$-algebra using $\alpha$. 
We remark that the Connes exact sequence in the lemma \ref{conexact} does not imply the desired non-vanishing property.

\section{Dualization}\label{sec:dual}
As $\AI$-algebras (resp. $\LI$-algebras) are given by coalgebras with codifferentials, the suitable dualization provides  non-commutative differential graded algebras (resp. commutative DGA) or a formal manifold in the language of Kontsevich and soibelman \cite{KS}. 
 This point of view is particularly interesting to
study homological algebras of these infinity algebras (see \cite{HL}) or homotopy cyclic infinity structures (see \cite{KS},\cite{C3}).

As mentioned in the introduction, in contact geometry, the dual language has been mostly used (\cite{Che}, \cite{EGH} for example) and
it also has a certain advantage as algebras can be easier to deal with than coalgebras.  But as we deal with Novikov fields, the dualization process is more complicated.

We  explain an appropriate procedure to take a dual of
a completed infinite-dimensional space over $\NOVE$.
We will work with $\NOVE$ in this section, as we would like to work with field coefficients (see \ref{def:nove}) for dualization.

Let $V$ be a vector space over the field $\NOVE$. Here, we assume $V$ have at most countably many generators $\{v_i\}_{i\in \NN}$ and $$V = \oplus_i (\NOVE <v_i>).$$ 
We will consider $V$ as a topological vector space
by defining a fundamental system of neighborhoods of $V$ at $0$:
first define the filtrations $F^{>\lambda}V$ as
$$F^{>\lambda}V = \{ \sum_{j=1}^k a_j v_{i_j}| a_i \in \NOVE,\tau(a_i) > \lambda, \; \forall i\}.$$
Here $\tau$ is the valuation of $\NOVE$ which gives the minimal exponent of $q$ defined in \ref{def:tau}. 
We regard $F^{>\lambda}V$ for $\lambda=0,1,2,\cdots$ as fundamental system of neighborhoods at $0$, and neighborhoods at $v \in V$ then are given by $v + F^{>\lambda}V$.

The completion of $V$ with respect to energy, $\HH{V}$, has been considered throughout the paper, and
it can be also considered as a completion using the Cauchy sequences in $V$ in this topological vector space
(see \cite{AM} for example). Let $\HH{F}^{>\lambda}V$ be the induced open set of $\HH{V}$ from $F^{>\lambda}V$ for each $\lambda$.

This topology has been introduced to consider the topological dual space $\HH{V}^*$ of $\HH{V}$.
We define $\HH{V}^*$ to be the set of all continuous $\NOVE$-linear maps  from $\HH{V}$ to $\NOVE$:
$$\HH{V}^* = Hom_{cont}(\HH{V}, \NOVE).$$
More explicitly we can describe $\HH{V}^*$ in the following way.
Denote by $v_i^* \in \HH{V}^*$ a map which is defined as $\NOVE$-linear extension of
$$v_i^*(v_i)=1, v_i^*(v_j)=0 \;\; \textrm{for } \; j\neq i.$$
The map $v_i^*$ is continuous and so is any finite sum of such $v_i^*$'s.

\begin{lemma}\label{dual}
For any $\lambda_0 \in \RR$, any map given by an infinite sum
$$v^* = \sum_{j=1}^\infty a_jv_{i_j}^*, \;\; \textrm{with}\;\; a_i \in \NOVE, \tau(a_i) > \lambda_0, \; \forall i,$$
is always  continuous.

Moreover,  a map given by an infinite sum whose $\tau(a_i)$'s are not bounded below, is not continuous.
\end{lemma}
\begin{proof}
Note that for any open set $F^{>\lambda} \NOVE$ of $\NOVE$,
we have $v^*(\HH{F}^{> \lambda - \lambda_0}V) \subset F^{>\lambda} \NOVE$. Hence
$v^*$ is continuous. For the second assertion, consider 
$w^* = \sum_{j=1}^\infty b_jv_{i_j}^*$ with
$ \tau(b_j) \to -\infty$ as $j \to \infty$.
Then, for a given open set $F^{>\lambda} \NOVE$ and for any $v \in \HH{V}$ and any $\lambda_1 \in \RR$,
we can find $y \in v + \HH{F}^{>\lambda_1}(V)$ such that $w^* (y) \notin F^{>\lambda} \NOVE$.
Such $y$ can be chosen for example as 
$$v+ \sum_{j=s}^\infty v_{i_j} q^{\lambda - \tau(b_j) -1},$$
where $s$ is any number with $(\lambda - \tau(b_j) -1 > \lambda_1)$ for all $j>s$.
One can find such $s$ as $(\lambda - \tau(b_j) -1)$ converges to infinity as $j \to \infty$.
\end{proof}

The above lemma explains what are the elements of $\HH{V}^*$. Intuitively, the dual elements are allowed to have
infinite sums with bounded energy since when we evaluate them with $\HH{V}$, the input  already
has energy converging to infinity in its infinite sum.

\section{Chevalley-Eilenberg cohomology}
In this section, we consider the dual of (cyclic) Chevalley-Eilenberg homologies, which we call
(cyclic) Chevalley-Eilenberg cohomology.  Then, we express the cochain complex in a more
explicit form and compare with the work of Cornea and Lalonde in \cite{CL}. 
We make computations of (cyclic) Chevalley-Eilenberg cohomology when the $\AI$-algebra
has non-vanishing primary obstruction cycle, and show the vanishing of cohomology using the natural algebra structure on them.

\subsection{Cyclic Chevalley-Eilenberg cohomology}
We apply the construction in the previous section to define the dual of the cyclic Chevalley-Eilenberg chain complexes introduced in the Definition \ref{def:cycho}.
Recall that we have a bar subcomplex $(\HH{E}C,\HH{d})$ over $\NOVO$-coefficient from section 2. We may change the coefficient of
$(\HH{E}C,\HH{d})$ to be $\NOVE$ and denote it again with the same notation.
 We assume that $C$ has at most countable generators. 
We regard $\HH{E}C$ as a topological vector space as in the previous subsection, and
take the topological dual
$$\HH{E}C^* := Hom_{cont}(\HH{E}C,\NOVE).$$
One can see that $\HH{d}^*$ also naturally defines a differential and
$(\HH{E}C^*,\HH{d}^*)$ forms a chain complex.

Recall that $\NOVE$ is a field. By following a standard proof of the universal coefficient theorem
(see for example \cite{DK}), we have 
\begin{lemma}\label{ceiso}
There exists a natural map from the homology of $(\HH{E}C^*,\HH{d}^*)$ to the topological dual of
the homology of $(\HH{E}C,\HH{d})$ which is an isomorphism.
$$H_\bullet(\HH{E}C^*,\HH{d}^*) \stackrel{\cong}{\longrightarrow} \big(H_\bullet(\HH{E}C,\HH{d}) \big)^*$$
The same statement holds for $(\HH{E}_{\geq 1}C,\HH{d})$ also.
\end{lemma}
In fact, the difference between $H_\bullet(\HH{E}C^*,\HH{d}^*)$ and $H_\bullet(\HH{E}_{\geq 1}C^*,\HH{d}^*)$ can be
easily seen as follows. In $\HH{E}C^*$, there exists the linear functional $\HH{E}C \to \NOVE$ given by the projection to the length zero component and hence identity on $E_0C = \NOVE$.
Note that we have a short exact sequence
$$0 \to (\HH{B}_{\geq 1} C,\HH{d}) \to (\HH{B} C,\HH{d}) \to (\NOVE,0) \to 0.$$
By considering its dual exact sequence and its associated long exact sequence, we have
$$0 \to H_1(\HH{E}C^*,\HH{d}^*) \to H_1(\HH{E}_{\geq 1}C^*,\HH{d}^*) 
\to \NOVE $$
$$\to  H_0(\HH{E}C^*,\HH{d}^*) \to H_0(\HH{E}_{\geq 1}C^*,\HH{d}^*) \to 0.$$
The generator of $\NOVE$ in the middle of the above will correspond to 
$H_1(\HH{E}_{\geq 1}C^*,\HH{d}^*)$ if $m_0(1) \neq \HH{d}(\alpha)$ for any $\alpha \in \HH{E}_{\geq 1}C$.
If $m_0(1) = \HH{d}(\alpha)$, we have a non-trivial element $(1 - \alpha) \in H_0(\HH{E}C^*,\HH{d}^*)$.

Despite the lemma \ref{ceiso}, we remark that there is an advantage to consider the cohomology theory in this case
as algebras are generally easier to work with than coalgebras and this will be essentially used to prove
the vanishing results later.  More precisely, from the Lemma \ref{lem:coalge}, we have 
$$\HH{E}C^* \otimes \HH{E}C^* \to (\HH{E}C \otimes \HH{E}C)^* \stackrel{\Delta^*}{\longrightarrow} \HH{E}C^*$$
This provides an algebra structure on $\HH{E}C^*$ with a unit $1$, where the unit is a map 
$\HH{E}C \to \NOVE$ which is identity on $E_0C = \NOVE$ and vanishes elsewhere.
For later arguments, it is essential to have a unit of commutative DGA.
Hence we will consider $(\HH{E}C^*,\HH{d}^*)$ mostly, and call it the extended cyclic Chevalley-Eilenberg cohomology.

Now, we express more explicitly the dual space $\HH{E}C^*$ with generators.
Suppose that the $\NOVE$-module $C$ has generators $\{e_i\}_{i\in I}$ where $I$ is
at most a countable set. We may also assume that the valuation $\tau(e_i)=0$ and $e_i$ is homogeneous 
of degree $|e_i|'$.
We write the dual $e_i^* = x_i$ and define the degree of $x_i$ as
$|x_i|' = - |e_i|'$. We may write 
$$[e_{i_1},\cdots,e_{i_k}]^* =x_{i_1}x_{i_2}\cdots x_{i_k}$$
where we define the variables $x_i$'s to be graded commutative:
$$x_i \cdot x_j = (-1)^{|x_i|'|x_j|'} x_j \cdot x_i.$$
We call the number of variables $x_i$'s in the monomial to be its length.

Consider the vector space $\RR<x_i>_{i\in I}$ generated by these variables and
consider also the free graded commutative algebra over the vector space $\RR<x_i>_{i\in I}$
and denote them by $S(\RR<x_i>)$, in which elements are given by finite sum of monomials of finite length.
By the lemma \ref{dual}, we can give the following definition.
\begin{definition}
We define the extended cyclic Chevalley-Eilenberg cochain $\HH{E}C^*$ alternatively as 
$$CE^\bullet(C) = \big(S(\RR<x_i>_{i\in I}) \otimes \NOVE \big)^\wedge$$
where in the completion $(\;)^\wedge$, we allow infinite sums with the valuations of its coefficients bounded from below.
Coboundary operation is given by $\HH{d}^*$, to define extended Chevalley-Eilenberg cohomology.
\end{definition}
We remark that Cornea and Lalonde has announced a cluster homology theory of Lagrangian submanifolds in \cite{CL}.
They have used the Morse function and gradient flows and allowed several disc components connected by Morse flows.
The construction of \cite{FOOO} is based on singular chains rather then Morse functions and gradient flows.
Here, the analogy is that one may think of singular chains as unstable manifolds of the given Morse function.

To obtain the actual cluster complex of \cite{CL}, one should take the topological dual of cyclic Chevalley-Eilenberg complex of
the following $\AI$-algebra recently constructed by Fukaya, Oh, Ohta and Ono.
\begin{theorem}[\cite{FOOO3} Theorem 5.1]
Let $L$ be a relatively spin Lagrangian submanifold in a closed symplectic manifold $(M,\omega)$. Then there exist
a Morse function $f$ such that the Morse complex $CM^*(f) \otimes \NOVO$ carries a structure of a filtered $\AI$-algebra, which
is homotopy equivalent to the filtered $\AI$-algebra constructed in \cite{FOOO}
\end{theorem}
Recall that the construction of the $\AI$-algebra in the above theorem is given by first constructing $A_{n,K}$-algebra
for each $(n,K)$ and for $(n,K) \prec (n',K')$, $A_{n,K}$-equivalence between such $A_{n,K}$ and $A_{n',K'}$-algebras.
From this, they construct $\AI$-algebra in a purely algebraic way, by pulling back higher $A_\infty$ structures.

Hence, the following comparison will only hold up to large $(n,K)$. Now, to construct $A_{n,K}$-algebra,
the chains $\chi_{g}$ are constructed inductively for $g \in \NN$ so that for any $(g_0)$, there exist
$g_1 >g_0$ such that they construct $A_{n,K}$ algebra structure on $\chi_{g}$ with the following properties.
Namely, $A_{n,K}$ structure is defined on $\chi_{g_0}$ in a geometric way (using Kuranishi perturbation and fiber products), and
then they are extended to $\chi_g$ algebraically using the sum over tree formula or
homological perturbation lemma.

In \cite{FOOO3}, they consider a specific choice of $f$ constructed in a way compatible with triangulation of $L$, so that
the sum over tree formula may be interpreted as counting gradient flow trees whose vertices represent pseudo-holomorphic discs and whose gradient flows
represent gradient flow lines (see their figure 6 of \cite{FOOO3}).
This is exactly as in the cluster complex case (where the only difference is the direction of flows).
Hence, by taking the dualization as in the previous subsection, the construction of \cite{FOOO3} becomes in fact quite
similar to that proposed by Cornea and Lalonde (for large $(n,K)$).  We refer readers to \cite{FOOO3} for more details on
their construction.

In any case, after taking the dual of \cite{FOOO3},  we obtain the completed symmetric algebra on generators and obtain differential graded commutative algebra (comm. DGA) as in \cite{CL}. Cornea and Lalonde also introduced symmetric fine Floer homology which
is defined for a pair, Lagrangian submanifold and its hamiltonian isotopy image. This correspond to the
Chevalley-Eilenberg cohomology for $\LI$-modules which will be explained in the next subsection.

But there is a subtlety regarding the filtrations. Namely, the filtration we use here is different from that of Cornea and Lalonde.
Here we recall their filtration of the cluster complex of \cite{CL} equation (1):
\begin{equation}\label{CLfilt}
L^k(S\QQ<Crit(f)[1]>\otimes \NOV> = Q<x_1x_2\cdots x_s e^{\lambda}:s \geq k \;or \;\omega(\lambda) \geq k>
\end{equation}
Hence infinite sums either have length of each term converging to infinity or energy converging to infinity with the
above filtration (\ref{CLfilt}). In particular, infinite sum of monomials whose length goes to infinity while energy converging to negative
infinity is allowed.

But in our case, due to the Lemma \ref{dual}, we do not allow  such infinite sums of unbounded negative energy.
And as we will see, this will cause different behaviors of resulting homology theories. 

\begin{remark}
We have been informed by Cornea that the filtration used here also can be used in the cluster homology theory, and
we thank him for his comments. But we do not know whether the filtration used in \cite{CL} can be used here to provide an invariant homology theory
as we prove the invariance before we take the dualization and then use Lemma \ref{ceiso}.
\end{remark}

We define $\WT{\tau}:CE^\bullet(C) \to \RR$ as in (\ref{def:tau}),
which gives the minimal exponent of $q$ used in the coefficients of an element in $CE^\bullet(C)$.
The product structure of $\HH{E}C^*$ corresponds to the natural product structure on
$CE^\bullet(C)$ which may be considered as a usual product of formal series of commuting variables.

As we work on DGA, we can use the clever argument from the work of Cornea and Lalonde:
\begin{prop}[cf. \cite{CL} Proposition 1.3]\label{lem:cevanish}
Suppose that for some $x\in CE^\bullet(C)$, we have
$$\HH{d}^*(x) = 1 + h, $$
for $h\in CE^\bullet(C)$ with $\WT{\tau}(h) \geq 0$ and $h$ has only terms with positive length. 
Then the homology of $(CE^\bullet(C),\HH{d}^*)$ vanishes
\end{prop}
\begin{remark}
The condition $\WT{\tau}(h) \geq 0$, which is rather restrictive, is not required in \cite{CL} due to
the different choice of filtration. 
\end{remark}
\begin{proof}
The condition on $h$ guarantees that the following is an element of  $CE^\bullet(C)$.
$$h'=\sum_{j=0}^\infty  (-1)^j h^j.$$
As $\HH{d}^* \circ \HH{d}^* =0$, and $\HH{d}^*(1)=0$, we have
$\HH{d}^* h=0$. As $\HH{d}^*$ is a derivation of the DGA $CE^\bullet(C)$,
we also have $\HH{d}^*h'=0$.
Hence, 
\begin{equation}
\HH{d}^*(x \cdot h') = (\HH{d}^*(x) \cdot h') = (1+h)(\sum_{j=0}^\infty  (-1)^j h^j) = 1
\end{equation}
As $1$ is a coboundary, this implies that any $\HH{d}^*$-cocycle $y \in CE^\bullet(C)$ is a coboundary.
$$y = 1 \cdot y = \HH{d}^*(x \cdot h') \cdot y = \HH{d}^*(x \cdot h' \cdot y).$$ 
\end{proof}

In the case that there is no quantum contribution from pseudo-holomorphic discs, one can compute the extended cyclic
Chevalley-Eilenberg cohomology easily.
First, recall the following theorem:
\begin{theorem}[Theorem X, \cite{FOOO}]\label{thm:derham}
In the case that there is no quantum contribution, the $\AI$-algebra of Lagrangian submanifold $(H^*(L,\RR),m)$ is homotopy equivalent to the de Rham complex of $L$ as an $\AI$-algebra.
\end{theorem}
\begin{corollary}
In the case that there is no quantum contribution, the extended cyclic Chevalley-Eilenberg cohomology is
isomorphic to 
$$(S(H_*(L,\RR)[1]) \otimes \NOVE)^\wedge.$$
\end{corollary}
\begin{proof}
Note that de Rham complex is a differential graded algebra, hence $\OL{m}_{k} \equiv 0$ for $k\geq 3$. 
And the product $\OL{m}_2$ is graded commutative. Hence $\OL{l}_k \equiv 0$ for $k \geq 2$.
Hence the $\AI$-algebra of Lagrangian submanifold in this case is homotopy equivalent to another $\AI$-algebra structure on the singular homology $H_*(L,\RR)$ whose induced $\LI$-structure is trivial. As the (extended) cyclic Chevalley-Eilenberg cohomology
is an invariant of the homotopy class, and all the differential vanish in the latter case, hence the claim follows.
\end{proof}

Now, we can prove the theorem stated in the introduction.
\begin{theorem}
Let $L$ be a relatively spin Lagrangian submanifold in a symplectic manifold $(M,\omega)$ with
 non-vanishing primary obstruction cycle. Let $A$ be the $\AI$-algebra of $L$.
 Then its extended cyclic Chevalley-Eilenberg cohomology vanishes.
\end{theorem}
\begin{proof}
 We will construct an element $x$ which satisfies
the assumption of the Proposition \ref{lem:cevanish}. 

We briefly recall the definition of a primary obstruction cycle.
We label $$0=\beta_0,\beta_1,\cdots,\beta_k,\cdots,$$ the equivalence classes of homotopy classes of pseudo-holomorphic discs with boundary on $L$, where two homotopy classes are equivalent if they have the same Maslov indices and symplectic energies.
Here enumeration is made so that $\omega(\beta_i) \leq \omega(\beta_{i+1})$ for a symplectic form $\omega$.

Suppose that $\lambda :=\omega(\beta_1)=\cdots = \omega(\beta_j) < \omega(\beta_{j+1})$ for some $j \geq 1$.
As we consider equivalence classes, we have 
$$\mu(\beta_s) \neq \mu(\beta_t)\; \textrm{for any} \;1 \leq s \neq t \leq j$$

As the classes $\beta_1,\cdots,\beta_j$ are minimal classes, the boundary image of holomorphic discs in the
class $\beta_s$, which is $m_{0,\beta_s}(1)$, defines a cycle of $\OL{m}_1$ for each
$s=1,\cdots,j$.
Primary obstruction cycles are defined as 
$\mathcal{O}_{s} = m_{0,\beta_s}(1)$ for each $s$.

Now we assume that we work on the canonical model $A_{can}$ of $A$, and the induced $\LI$-algebra structure 
$\WT{A}_{can}$ is trivial as in the above corollary. This means that we have $l_{k,\beta_0} = \OL{l}_k \equiv 0$.

The $\OL{m}_1$-cycle $\mathcal{O}_s$ is non-trivial and in the canonical model, we still have $m_{0,\beta_s}(1) = \mathcal{O}_s$.
Here we may work on the canonical model as we have proved that extended cyclic Chevalley-Eilenberg cohomology is an invariant of homotopy class of $A$.

We set $x_s$ to be a dual variable to $\mathcal{O}_s$ in $CE^\bullet(C)$.
Then, 
$$\HH{d}^* x_s (1) = x_s \big(\HH{d}(1) \big) =  x_s \big( m_{0,\beta_1}T^{\lambda}e^{\mu(\beta_1)} + 
\cdots m_{0,\beta_1}T^{\lambda}e^{\mu(\beta_s)}) + \textrm{higher energy terms}$$
$$= 1 \cdot T^{\lambda}(e^{\mu(\beta_s)} + \xi) + T^{\lambda}\eta =: a_0 \cdot T^{\lambda}.$$
Here as $m_{0,\beta_t}$ may have non-trivial $x_s$ value for $t \neq s$, hence we write 
such contribution as $\xi$, where there cannot be any cancellation as each $\mu(\beta_t)$ is distinct.
And by $\eta \in F^{>0}\NOVE$, we denote the rest with higher energy.
Clearly, $$a_0T^{\lambda} = \HH{d}^*x_s(1) \neq 0.$$
Note that $a_0$ is invertible and consider its inverse $1/a_0$.
Consider $$y = \frac{1}{a_0}T^{-\lambda} x_s.$$
Then, we have $\HH{d}^* y(1) = 1$ by definition. Hence
we have $$\HH{d}^* y = 1 + h,$$
where $h$ has terms of positive length.
Also note that we have $\WT{\tau}(h) \geq 0$ because $\lambda$ is the minimal energy with non-trivial $\LI$-algebra operation.
Hence, $y$ satisfies the assumption of the lemma \ref{lem:cevanish} and implies the desired vanishing property.
\end{proof}

We remark that the related Proposition 1.3 in \cite{CL} is somewhat different due to a different choice of filtration \ref{CLfilt}.
It seems that in such a case it does not recognize the unobstructedness as in the above Theorem. Assume in the above proof that we work in the
chain level (not in the canonical model) and suppose all the primary obstructions vanish. i.e. there exists a chain $b_s$
with $\OL{m}_1(b_s) = - \mathcal{O}_s$. Let us assume that $m_{0,\beta_s}$ is a chain which is not zero. 
(i.e. $b_s$ is not zero). Then, consider a dual variable $x_s$ of $m_{0,\beta_s}$. Note that $m_{0,\beta_s}$ is homologically
trivial, but as we take dual on the chain level we have a corresponding dual variable. Then we have as before
$$\HH{d}^* x_s (1) = a_0T^{\lambda},$$
for a non-trivial $a_0$ with $\tau(a_0)=1$.
But also 
$$\HH{d}^* x_s( b_s) = x_s(m_1(b_s)) = x_s(\OL{m}_1(b_s)) + \; \textrm{higher energy terms}.$$
Here we have
$$x_s(\OL{m}_1(b_s)) = x_s(-\mathcal{O}_s) = -1.$$
If we denote the dual variable of $b_s$ to be $x_s'$, then we have 
$$\HH{d}^* x_s = a_0 T^{\lambda} - x_s'+ h,$$
for some $h$. Hence, $\HH{d}^*(T^{-\lambda}x_s)$ will have a component
$-T^{-\lambda}x_s'$ which has negative energy.

Recall that in the proof of the lemma \ref{lem:cevanish}, one takes $\sum_{j=0}^\infty (-1)^j (T^{-\lambda}x_s')^j$ which
would have unbounded negative energy.
With the filtration (\ref{CLfilt}) of \cite{CL}, such an expression is allowed 
and it will prove the  vanishing of the homology. 

But in the case of our paper, such an expression with unbounded energy is not allowed
and hence such an argument cannot be used to prove the vanishing of 
Chevalley-Eilenberg cohomology.

\subsection{Chevalley-Eilenberg cohomology}
Similarly, we take the topological dual of the Chevalley-Eilenberg chain complexes defined in the section \ref{sec:CEho}
for $\LI$-modules $M$ over $\LI$-algebra $\WT{A}=(C,l)$, and call its homology a Chevalley-Eilenberg cohomology $CE^\bullet(\WT{A},M)$. In fact we will only consider the case $M=\WT{A}$. 
By proceeding as in the previous subsection, we obtain

\begin{definition}
We define the Chevalley-Eilenberg cochain $\big(CE_\bullet(\WT{A},\WT{A})\big)^*$ alternatively as
$$CE^\bullet(C,C) = C \otimes CE^\bullet(C)= C\otimes  (S\RR<x_i>_{i\in I} \otimes \NOVE)^\wedge.$$
Coboundary operation is given by $(d^{CE})^*$, to define the Chevalley-Eilenberg cohomology.
\end{definition}
By proceeding as in the standard universal coefficient theorem, one can prove that
\begin{lemma}
There exists a natural map from the homology of $(CE^\bullet(\WT{A},\WT{A}),(d^{CE})^*)$ to the topological dual of
the homology of $(CE_\bullet(\WT{A},\WT{A}),d^{CE})$ which is an isomorphism.
\end{lemma}
\begin{corollary}
If $L$ is displaceable from itself via Hamiltonian isotopy, its Chevalley-Eilenberg cohomology vanishes.
\end{corollary}
\begin{proof}
This follows from the corollary \ref{isoiso2} and the above lemma.
\end{proof}
Now, we can prove the remaining part of the theorem 1.1.
\begin{theorem}
If a Lagrangian submanifold $L$ has a non-trivial primary obstruction class, then its Chevelley-Eilenberg cohomology vanishes.
\end{theorem}
\begin{proof}
This proceeds as in the remark 1.11 of \cite{CL}. Namely, one can see that $CE^\bullet(C,C)$ has a differential graded right module structure over a differential graded algebra $CE^\bullet(C)$ (which is obtained as a dual of a comodule).
 Hence when the homology of $CE^\bullet(C)$ is trivial, it is easy to show that  the homology of $(CE^\bullet(C,C),(d^{CE})^*)$ is also trivial.
\end{proof}

\bibliographystyle{amsalpha}

\end{document}